\documentclass[ejs]{imsart}

\RequirePackage[OT1]{fontenc}
\RequirePackage{amsthm,amsmath}
\RequirePackage[numbers]{natbib}
\RequirePackage[colorlinks,citecolor=blue,urlcolor=blue]{hyperref}
\usepackage{mystyle}

\usepackage{algorithm}
\usepackage{algpseudocode}

\newcommand{\eps}{\varepsilon}
\newcommand{\nint}[1]{[#1]}
\newcommand{\G}{\Gg}
\newcommand{\E}{E}
\newcommand{\V}{V}
\newcommand{\Inc}{\mathbf{D}}
\newcommand{\Inct}{{\mathbf{D}^\top}}
\newcommand{\Lap}{\mathbf{L}}

\newcommand{\x}{\beta}
\newcommand{\xs}{\beta^\star}
\newcommand{\xe}{\hat{\beta}}

\newcommand{\norms}[1]{\norm{#1}_{[\lambda]}}

\usepackage{xspace}
\newcommand{\GSlope}{Graph-Slope\xspace}
\newcommand{\Slope}{Slope\xspace}

\newcommand{\GLasso}{Graph-Lasso\xspace}
\graphicspath{{images/},{prebuiltimages/}}

\begin{document}

\begin{frontmatter}
\title{A sharp oracle inequality for \GSlope}
\runtitle{A sharp oracle inequality for \GSlope}

\begin{aug}
  \author{\fnms{Pierre C} \snm{Bellec}
    \ead[label=e1]{pcb71@stat.rutgers.edu}
  \ead[label=u1,url]{http://www.stat.rutgers.edu/home/PCB71/}}
\address{Department of Statistics \& Biostatistics,\\
  Rutgers, The State University of New Jersey,\\
  501 Hill Center, Busch Campus,\\
  110 Frelinghuysen Road, Piscataway, NJ 08854,\\
  \printead{e1}\\
  \printead{u1}}

\and

\author{\fnms{Joseph} \snm{Salmon}\thanksref{t1}
  \ead[label=e2]{first.lastname@telecom-paristech.fr}
\ead[label=u2,url]{http://josephsalmon.eu}}
\address{LTCI, T\'el\'ecom ParisTech,\\
Universit\'e Paris-Saclay, 75013, Paris,\\
\printead{e2}\\
\printead{u2}}

\and

\author{\fnms{Samuel} \snm{Vaiter}\thanksref{t1}
  \ead[label=e3]{samuel.vaiter@u-bourgogne.fr}
  \ead[label=u3,url]{http://samuelvaiter.com}}
\address{CNRS \& IMB, Universit\'e de Bourgogne,\\
9 avenue Alain Savary, 21000, Dijon,\\
\printead{e3}\\
\printead{u3}}

\thankstext{t1}{J. Salmon and S. Vaiter were supported by the CNRS PE1 ``ABS'' grant.}
\runauthor{Bellec et al.}

\affiliation{Rutgers and Telecom ParisTech and CNRS}

\end{aug}

\begin{abstract}

Following recent success on the analysis of the Slope estimator, we provide a sharp oracle inequality in term of prediction error for Graph-Slope, a generalization of Slope to signals observed over a graph. In addition to improving upon best results obtained so far for the Total Variation denoiser (also referred to as Graph-Lasso or Generalized Lasso), we propose an efficient algorithm to compute Graph-Slope.
The proposed algorithm is obtained by applying the forward-backward method to the dual formulation of the Graph-Slope optimization problem.
We also provide experiments showing the practical applicability of the method.

\end{abstract}


\tableofcontents
\end{frontmatter}


\section{Introduction}
\label{sec:intro}

Many inference problems of interest involve signals defined on discrete graphs.
This includes for instance two-dimensional imaging but also more advanced hyper-spectral imaging scenarios where the signal lives on a regular grid.
Two types of structure arise naturally in such examples:
The first type of structures comes from regularity or smoothness of the signal, which led to the development of wavelet methods.
The second type of structure involves signals with few sharp discontinuities.
For instance in one dimension, piecewise constant signals appear when transition states are present, the graph being a 1D path.
In imaging, where the underlying graph is a regular 2D grid, occlusions create piece-wise smooth signals rather than smooth ones.

This paper studies regularizers for signals with sharp discontinuities.
A popular choice in imaging is the Total Variation (TV) regularization \cite{Rudin_Osher_Fatemi92}.
For 1D signals, TV regularization has also long been used in statistics \cite{Mammen_vandeGeer}. If an additional $\ell_1$ regularization is added, this is sometimes referred to as the fused Lasso \cite{Tibshirani_Saunders_Rosset_Zhu_Knight05,viallon2016robustness,Dalalyan_Hebiri_Lederer17}.

A natural extension of such methods to arbitrary graphs relies on $\ell_1$ analysis penalties \cite{Elad_Milanfar_Rubinstein07}
which involve the incidence matrix of the underlying graph,
see for instance \cite{sadhanala2016total} or the Edge Lasso of \cite{sharpnack2012sparsistency}.
Such penalties have the form
\begin{equation}
    \text{pen }: \x\rightarrow \lambda \| \Inct\x \|_1\enspace,
\end{equation}
where $\lambda > 0$ is a tuning parameter and $\Inct $ is the (edge-vertex) incidence matrix of the graph defined below, and $\beta$ represents the signal to be recovered.
This approach is notably different from contributions in machine learning where $\ell_2$ penalties, \ie Laplacian regularization, have been considered for spectral clustering \cite{shi2000normalized,ng2001spectral}  (see also \cite{von2007tutorial} for a review). Theoretical results in favor of the $\ell_1$ norm instead of the squared $\ell_2$ norm are studied in \cite{sadhanala2016total}.

Penalties based on $\ell_0$ regularization with the graph incidence matrix have recently been analyzed \cite{fan2017pw}, including an analysis of their approximation algorithm. They are of interest as they do not suffer from the (shrinkage) bias created by the convex $\ell_1$ norm. However, such methods present the difficulty that in the general case they lead to non-convex problems.
Note that the 1D path is an exception since the associated optimization problem can be solve using dynamic programming~\cite{auger1989algorithms}.
Concerning the bias reduction though, simpler remedies could be used, including least-squares refitting on the model space associated, applying for instance the CLEAR method \cite{Deledalle_Papadakis_Salmon_Vaiter16}.

Following the introduction of the Slope regularization in the context of high dimensional regression \cite{bogdan2015slope}, we propose \GSlope, its generalization to contexts where the signal is supported on a graph.
In linear regression, Slope \cite{bogdan2015slope} is defined as follows.
Given $p$ tuning parameters $\lambda_1\ge\lambda_2\ge\dots\ge\lambda_p\ge 0$ with at least one strict inequality,
define the ordered $\ell_1$ norm by
\begin{equation}
  \norms{\theta} = \sum_{j=1}^p \lambda_j |\theta|_j^\downarrow \enspace,
  \label{eq:def-ordered-ell-1-norm}
\end{equation}
where for any $\theta\in  \bbR^p$, we use the notation\footnote{following the notation considered in \cite{Bhatia97}} $(|\theta|_1^\downarrow,\dots,|\theta|^\downarrow_p)$ for the non-increasing rearrangement of its amplitudes $(|\theta_1|,\dots,|\theta_p|) $.
Then, given a design matrix $X\in\bbR^{n\times p}$ and a response vector $y\in\bbR^n$, the Slope estimator is defined as a solution of the
minimization problem
\begin{equation}
    \min_{b\in\bbR^p} \frac{1}{2n}\|y - X\beta\|^2 + \norms{\beta} \enspace.
\end{equation}
If the parameters $\lambda_1,\dots,\lambda_p$ are all equal, then Slope is equal to the Lasso with tuning parameter $\lambda_1$.

Slope presents several advantages compared to Lasso in sparse linear regression.
First, Slope provably controls the False Discovery Rate (FDR) for orthogonal design matrices \cite{bogdan2015slope}
and experiments show that this property is also satisfied for some non-orthogonal design matrices \cite{bogdan2015slope}.
Second, it appears that Slope has more power than Lasso in the sense that Slope will discover more nonzero coefficients of the unknown target vector \cite{bogdan2015slope}. An interpretation of this phenomenon is that Lasso shrinks small coefficients too heavily and may thus miss the smallest nonzero coefficients of the target vector.
On the other hand, the Slope penalty induces less shrinkage on small coefficients, leading to more power.
Third, while Lasso with the universal parameter is known to achieve the rate of estimation of order $(s/n)\log(p)$ (where $s$ is the sparsity of the unknown target vector, $n$ the number of measurements and $p$ the number of covariates), Slope achieves the optimal rate of estimation
of order $(s/n)\log(p/s)$ \cite{su2016slope,bellec2016slope}.

We propose a theoretical and experimental analysis of \GSlope, the
counterpart estimator of Slope for signals defined on graphs. \GSlope is
defined in the next section.
Our theoretical contribution for \GSlope borrows some technical details recently introduced in \cite{hutter2016optimal} to control the Mean Squared Error (MSE) for the Generalized Lasso.

Last but not least, we provide an efficient solver to compute the \GSlope estimator. It relies on accelerated proximal gradient descent to solve the dual formulation \cite{Beck_Teboulle09,Combettes_Pesquet11,Parikh_Boyd_Chu_Peleato_Eckstein13}.
To obtain an efficient solver, we leverage the seminal contribution made in \cite{Zeng_Figueiredo14} showing the link between ordered $\ell_1$ norm \eqref{eq:def-ordered-ell-1-norm} and isotonic regression. Hence, we can use fast implementations of the PAVA algorithm (for Pool Adjacent Violators Algorithm, see for instance \cite{best1990active}), available for instance in \texttt{scikit-learn} \cite{Pedregosa_etal11} for this purpose.
Numerical experiments illustrate the benefit of \GSlope, in particular in terms of True Discovery Rate (TDR) performance.

A high level interpretation of our simulation results is as follows.
In the model considered in this paper, a sharp discontinuity of the signal corresponds to an edge of the graph with nonzero coefficient.
Since Graph-Lasso uses an $\ell_1$-penalty, the penalty level is uniform across all edges of the graph.
Edges with small coefficients are too heavily penalized with Graph-Lasso.
Using \GSlope lets us reduce the penalty level on the edges with small coefficients. This leads to the discovery of more discontinuities of the true signal as compared to Graph-Lasso.

\subsection{Model and notation}
\label{sub:notation}

Let $\G = (\V,\E)$ be an undirected and connected graph on $n$ vertices, $\V = \nint{n}$, and $p$ edges, $\E = \nint{p}$.
This graph can be represented by its edge-vertex incidence matrix $\Inct = \Inc_{\G}^\top \in \RR^{p \times n}$ (we drop the reference to $\G$ when no ambiguity is possible) defined as
\begin{equation}
  (\Inct)_{e,v} =
  \begin{cases}
    + 1, & \text{if } v = \min(i,j) \\
    -1, & \text{if } v = \max(i,j) \\
    0, & \text{otherwise}
  \end{cases}\enspace,
\end{equation}
where $e = \ens{i,j}$.
The matrix $\Lap=\Inc \Inct$ is the so-called graph Laplacian of $\G$. The Laplacian $\Lap$ is invariant under a change of orientation of the graph.

For any $u\in\bbR^p$, we denote by $\norm{u}_0$ the pseudo $\ell_0$ norm of $u$ : $\norm{u}_0=\abs{\{j \in \nint{p}: u_j \neq 0\}}$, and for any matrix $\mathbf{A}$, we denote by $\mathbf{A}^{\dagger}$ its Moore-Penrose pseudo-inverse. The canonical basis of $\bbR^p$ is denoted $(e_1,\dots,e_p)$.

For any norm $\norm{\cdot}$ on $\bbR^n$,
the associated \emph{dual norm} $\norm{\cdot}^*$ reads at $v \in \bbR^n$
\begin{equation}
  \norm{v}^* = \usup{\norm{\beta} \leq 1} \dotp{v}{\beta} \enspace.
\end{equation}
As a consequence, for every $(\beta, v) \in \bbR^n \times \bbR^n$, one has $\dotp{\beta}{v} \leq \norm{\beta} \norm{v}^*$.

In this work, we consider the following denoising problem for a signal over a graph. Assume that each vertex $i\in[n]$ of the graph
carries a signal $\xs_i$. For each vertex $i\in[n]$ of the graph, one observes $y_i$, a noisy perturbation of $\xs_i$.
In vector form, one observes the vector $y\in\bbR^n$ and aims to estimate $\xs\in\bbR^n$, \ie
\begin{equation}\label{eq:model}
  y = \xs + \eps \enspace,
\end{equation}
where $\eps\sim \Nn(0, \sigma^2 \Id_n)$ is a noise vector.
We will say that an edge $e=\{i,j\}$ of the graph carries the signal $(\Inct\xs)_e$.
In particular, if two vertices $i$ and $j$ are neighbours and if they carry the same value of the signal, \ie $\xs_i = \xs_j$, then the corresponding edge $e=\{i,j\}$ carries the constant signal.
The focus of the present paper is on signals $\xs$ that have few discontinuities. A signal $\xs\in\bbR^n$ has few discontinuities if $\Inct\xs$ has few nonzero coefficients, 
\ie $\norm{\Inct\xs}_0$ is small,
or equivalently if most edges of the graph carry the constant signal.
In particular, if $\norm{\Inct\xs}_0 = s$, we say that $\xs$ is a vector of $\Inct$-sparsity $s$.

\subsection{The \GSlope estimator}
\label{sub:estimator}

We consider in this paper the so-called \emph{\GSlope} variational scheme:
\begin{equation}\label{eq:gslope}
  \xe:=\hat{\beta}^{\rm GS}  \in \uargmin{\x \in \bbR^p}
  \frac{1}{2n}\norm{y - \x}^2 + \norms{\Inct \x} \enspace,
\end{equation}
where 
\begin{equation}
  \norms{\Inct \x} = \sum_{j=1}^p \lambda_j |\Inct\x|_j^\downarrow \enspace,
\end{equation}
with $\lambda=(\lambda_1,\dots,\lambda_p)\in\bbR^p$ satisfying $\lambda_1 \geq \lambda_2 \geq \cdots \geq \lambda_p\ge0$, and using for any vector $\theta \in  \bbR^p$ the notation\footnote{following the notation from \cite{Bhatia97}.} $(|\theta|_1^\downarrow,\dots,|\theta|^\downarrow_p)$ for the non-increasing rearrangement of its amplitudes $(|\theta_1|,\dots,|\theta_p|) $.
According to~\cite{bogdan2015slope}, $\norms{\cdot}$ is a norm over $\bbR^p$ if and only if $\lambda_1 \geq \lambda_2 \geq \cdots \geq \lambda_p\ge0$ with at least one strict inequality.
This is a consequence of the observation that if $\lambda_1 \geq \lambda_2 \geq \cdots \geq \lambda_p\ge0$ then one can rewrite the \Slope-norm of $\theta$ as the maximum over all $\tau \in \mathfrak{S}_p$ (the set of permutations over $\nint{p}$), of the quantity $\sum_{i=1}^p \lambda_j |\theta_{\tau(j)}|$:
\begin{equation}
  \norms{\theta} = \umax{\tau \in \mathfrak{S}_p} \sum_{j=1}^p \lambda_j |\theta_{\tau(j)}| = \sum_{j=1}^p \lambda_j |\theta|^\downarrow_{j}\enspace.
\end{equation}

The Generalized Lasso (also sometimes referred to as TV denoiser) relies on $\ell_1$ regularization.
It was recently investigated in \cite{hutter2016optimal,sadhanala2016total}, and can be defined as
\begin{equation}\label{eq:glasso}
   \hat{\beta}^{\rm GL}  \in \uargmin{\x \in \bbR^p}
  \frac{1}{2n}\norm{y - \x}^2 + \lambda_1 \norm{\Inct \x}_1 \enspace,
\end{equation}
where $\norm{\cdot}_1$ is the standard $\ell_1$ norm, and $\lambda_1>0$ is a tuning parameter.

If $\lambda_1 = \lambda_2 = \dots = \lambda_p$ then $\norms{\theta} = \lambda_1 \norm{\theta}_1$ for all $\theta\in\bbR^p$,
so that the minimization problems \eqref{eq:glasso} and \eqref{eq:gslope} are the same.
On the other hand, if $\lambda_j>\lambda_{j+1}$ for some $j=1,\dots,p-1$, then the optimization problems \eqref{eq:glasso} and \eqref{eq:gslope} differ. For instance, if $\lambda_1 > \lambda_2>0$, all coefficients of $\Inct\x$ are equally penalized in the Graph-Lasso \eqref{eq:glasso},
while coefficients of $\Inct\x$ are not uniformly penalized in the \GSlope optimization problem \eqref{eq:gslope}.
Indeed, in the \GSlope optimization problem \eqref{eq:gslope},
the largest coefficient of $\Inct\x$ is penalized as in \eqref{eq:glasso}
but smaller coefficients of $\Inct\x$ receive a smaller penalization.
The \GSlope optimization problem \eqref{eq:gslope} is more flexible than \eqref{eq:glasso} as it allows
the smaller coefficients of $\Inct\x$ to be less penalized than its larger coefficients.
We will see in the next sections that this flexibility brings advantages to both the theoretical properties of $\hat{\beta}^{\rm GS}$
as well as its performance in simulations, as compared to $\hat{\beta}^{\rm GL}$.


\section{Theoretical guarantees: sharp oracle inequality}
\label{sec:theoretical_guarantees}


We can now state the main theoretical result of the paper, a sharp oracle inequality for the \GSlope.
For any integer $s$ and weights $\lambda=(\lambda_1,\dots,\lambda_p)$, define
\begin{equation}\label{eq:def-Lambda}
    \Lambda(\lambda, s) = \Big( \sum_{j=1}^s \lambda_j^2 \Big)^{1/2} .
\end{equation}

\begin{thm}\label{thm:oracle-inequality}
    Assume that the \GSlope weights $\lambda_1\ge\dots\ge\lambda_p\ge0$ are such that the event
    \begin{equation}
        \label{eq:event-dual-norm}
        \frac{1}{n}\norms{\Inc^\dagger \eps }^* \le 1/2
    \end{equation}
    has probability at least $1/2$.
    Then, for any $\delta\in(0,1)$, we have with probability at least $1-2\delta$
    \begin{equation}
        \frac 1 n \norm{\xe - \x^\star}^2
        \le
        \inf_{s\in[p]} \!
        \left[
            \inf_{\substack{\x\in\bbR^n\\\norm{\Inct\x}_0\le s}}
            \frac 1 n \norm{\x - \x^\star}^2
            +
            \frac{1}{2n}
            \left(
                \frac{3n\Lambda(\lambda,s)}{2\kappa(s)} + \frac{\sigma + 2\sigma\sqrt{2\log(1/\delta)}}{\sqrt n}
            \right)^2
        \right]\enspace,
        \label{eq:oracle-ineq}
    \end{equation}
    where $\Lambda(\cdot,\cdot)$ is defined in \eqref{eq:def-Lambda} and the compatibility factor $\kappa(s)$ is defined as
    \begin{equation}\label{eq:def-kappa}
        \kappa(s) \triangleq
        \inf_{v\in\bbR^n:
            3 \Lambda(\lambda,s)\norm{\Inct v}_2
            >
            \sum_{j=s+1}^p \lambda_j |\Inct v|_j^\downarrow
        }
        \left(\frac{\norm{v}}{\norm{\Inct v}_2} \right)\enspace.
    \end{equation}
  \end{thm}

\begin{proof}
    Let $\x$ be a minimizer of the right hand side of \eqref{eq:oracle-ineq} and let $s=\norm{\Inct\x}_0$.
    Define the function $f(\cdot)$ by
    \begin{equation}
        f(e) = \sup_{v\in\bbR^n: \norm{v} = 1}
        \left[
            e^\top v
            + n \Lambda(\lambda,s)\norm{\Inct v}_2
            -n \sum_{j=s+1}^p \lambda_j |\Inct v|_j^\downarrow
        \right].
    \end{equation}
    Let also $w = (\xe - \x)/(\norm{\xe - \x})$.
    By \Cref{lem:strong-convexity} and \Cref{lem:algebra-slope-norm} with $\alpha = 0$,
    for any $z \in \bbR^p$, we have
    \begin{align}
        \frac 1 2 (\norm{\xe - &z}^2
        - \norm{\x - z}^2
        + \norm{\xe - \x}^2
        )\\
        \le&
        \eps^\top(\xe - \x) + n \norms{\Inct \x} - n \norms{\Inct \xe}, \\
        \le&
        \eps^\top(\xe - \x) + n \Lambda(\lambda,s)\norm{\Inct (\xe-\x )}_2
        - n \sum_{j=s+1}^p \lambda_j |\Inct (\xe - \x)|_j^\downarrow, \\
        =&
        \norm{\xe - \x}\Big(
            \eps^\top w + n\Lambda(\lambda,s)\norm{\Inct w}_2
            -n \sum_{j=s+1}^p \lambda_j |\Inct w|_j^\downarrow
        \Big)
        , \\
        \le &
        \norm{\xe - \x} f(\eps)
        \le \frac 1 2 f(\eps)^2 + \frac 1 2 \norm{\xe - \x}^2,
    \end{align}
    where for the last inequality we used the elementary inequality $2ab\le a^2+b^2$.

    By \Cref{lem:Pi} we have $\Id_n = (\Id_n - \Pi) + \Pi$
    with $\Id_n - \Pi = (\Inct )^\dagger \Inct $ and
    where $\Pi$ is the orthogonal projection onto $\ker(\Inct )$.
    Furthermore, $\ker(\Inct )$ has dimension 1 so that $\norm{\Pi\eps}_2^2/\sigma^2$
    is a $\chi^2$ random variable with 1 degree of freedom.
    Thus for any $v\in\bbR^n$ with $\norm{v}=1$ we have
    \begin{equation}
        \eps^\top v
        =
        \eps^\top\Pi v
        +\eps^\top(\Id_n - \Pi) v
        \le
        \norm{\Pi\eps}_2
        +\eps^\top(\Id_n - \Pi) v.
        \label{eq:after-cs}
    \end{equation}
    Let us define the function $g(\cdot)$ by
    \begin{equation}
        g(e) = \sup_{v\in\bbR^n: \norm{v} = 1}
        \left[
            e^\top(\Id_n - \Pi) v
            + n \Lambda(\lambda,s)\norm{\Inct v}_2
            -n \sum_{j=s+1}^p \lambda_j |\Inct v|_j^\downarrow
        \right].
    \end{equation}
    Then, by the definition of $f,g$ and \eqref{eq:after-cs}, we have almost surely $f(\eps) \le \norm{\Pi\eps}_2 + g(\eps)$.
    By a standard bound on $\chi^2$ random variable with 1 degree of freedom, we have
    $\bbP(\norm{\Pi\eps}_2 \le \sigma+ \sigma\sqrt{2\log(1/\delta)})\ge 1-\delta$.
    Furthermore,
    the function $g$ is is 1-Lipschitz and $\eps\sim \mathcal{N}(0,\sigma^2\Id_n)$. By
    the Gaussian concentration theorem \cite[Theorem 10.17]{boucheron2013concentration}, we have
    \begin{equation}
        \bbP \left( g(\eps) \le \Med[g(\eps)] + \sigma\sqrt{2\log(1/\delta)} \right) \ge 1-\delta \enspace,
    \end{equation}
    where $\Med[g(\eps)]$ is the median of the random variable $g(\eps)$.
    Combining these two probability bounds with the union bound, we obtain
    $f(\eps) \le \Med[g(\eps)] + \sigma + 2\sigma\sqrt{2\log(1/\delta)}$
    with
    probability at least $1-2\delta$.

    To complete the proof, it remains to show that
    \begin{equation}
        \Med[g(\eps)]\le (3n/2) \Lambda(\lambda,s) / \kappa(s) \enspace.
        \label{eq:goal-median-1}
    \end{equation}
    By definition of the median, it is enough to show that
    \begin{equation}
        \bbP(g(\eps)\le (3n/2)\Lambda(\lambda,s)/\kappa(s))\ge 1/2.
        \label{eq:goal-median-2}
    \end{equation}
    By \Cref{lem:Pi} and the fact that
    $\Id_n - \Pi = (\Inct )^\dagger \Inct $, we obtain that for all $v$,
    \begin{align}
        \eps^\top(\Id_n - \Pi) v
        &=
        \eps^\top (\Inct )^\dagger \Inct v\\
        &\le
        \norms{((\Inct )^\dagger)^\top \eps }^*
        \norms{\Inct  v}
        =
        \norms{\Inc^\dagger \eps }^*
        \norms{\Inct  v}\enspace,
    \end{align}
    where we used
    the duality between $\norms{\cdot}^*$ and $\norms{\cdot}$
    for the second term and the fact that the transpose and the Moore-Penrose
    pseudo-inverse commute, which implies
    $(\Inc^\dagger)^\top  = \Inct^\dagger$.

    We now bound $g(\eps)$ from above on the event \eqref{eq:event-dual-norm}.
    On the event \eqref{eq:event-dual-norm},
    \begin{align}
        g(\eps)
        &\le
        \sup_{v\in\bbR^n: \norm{v} =1}
        \left[
            \frac n 2 \norms{\Inct  v}
            + n \Lambda(\lambda,s)\norm{\Inct  v}_2
            -n \sum_{j=s+1}^p \lambda_j |\Inct v|_j^\downarrow
        \right] \\
        &\le
        \sup_{v\in\bbR^n: \norm{v} =1}
        \left[
            \frac n 2 \sum_{j=1}^s \lambda_j |\Inct v|_j^\downarrow
            + n \Lambda(\lambda,s)\norm{\Inct v}_2
            - \frac n 2 \sum_{j=s+1}^p \lambda_j |\Inct v|_j^\downarrow
        \right] \\
        &\le
        \frac n 2
        \sup_{v\in \bbR^n: \norm{v} =1}
        \left[
            3 \Lambda(\lambda,s)\norm{\Inct v}_2
            - \sum_{j=s+1}^p \lambda_j |\Inct v|_j^\downarrow
        \right] \enspace.
    \end{align}
    Consider $v\in\bbR^n$ such that $\norm{v} = 1$ and
    $3 \Lambda(\lambda,s)\norm{\Inct v}_2 >\sum_{j=s+1}^p \lambda_j |\Inct v|_j^\downarrow$.
    Then, by the definition of $\kappa(s)$ given in defined in~\eqref{eq:def-kappa}
    we have
    \begin{equation}
            3 \Lambda(\lambda,s)\norm{\Inct v}_2
            - \sum_{j=s+1}^p \lambda_j |\Inct v|_j^\downarrow
            \le
            3 \Lambda(\lambda,s) \norm{v} /\kappa(s)
            =
            3 \Lambda(\lambda,s) / \kappa(s).
    \end{equation}
    Consider $v\in\bbR^n$ such that $\norm{v} = 1$ and
    $3 \Lambda(\lambda,s)\norm{\Inct v}_2 \le \sum_{j=s+1}^p \lambda_j|\Inct v|_j^\downarrow$,
    then we have trivially
    \begin{equation}
            3 \Lambda(\lambda,s)\norm{\Inct v}_2
            - \sum_{j=s+1}^p \lambda_j |\Inct v|_j^\downarrow
            \le
            0
            \le
            3 \Lambda(\lambda,s) / \kappa(s)\enspace.
    \end{equation}
    Thus, we have proved that on the event \eqref{eq:event-dual-norm} that has probability at least $1/2$, 
        \eqref{eq:goal-median-2} holds.
        This implies \eqref{eq:goal-median-1} by definition of the median.
\end{proof}
The constant $\kappa(s)$ is sometimes referred to as the compatibility factor of $\Inct$.
Bounds on the compatibility factor are obtained for a large class of random and deterministic graphs \cite{hutter2016optimal}.
For instance, for graphs with bounded degree, the compatibility factor is bounded from below (see for instance \cite[Lemma~3]{hutter2016optimal}).
In linear regression, constants that measure the correlations of the design matrix have been proposed to study the Lasso and the Dantzig selector:
\cite{bickel2009simultaneous} defined the Restricted Eigenvalue constant,
\cite{vandegeer2009conditions} defined the Compatibility constant,
\cite{ye2010rate} defined the Cone Invertibility factors and
\cite{Dalalyan_Hebiri_Lederer17} defined the Compatibility factor, to name a few.
The Weighted Restricted Eigenvalue constant was also defined in \cite{bellec2016slope} to study the Slope estimator.
These constants are the linear regression analogs of $\kappa(s)$ defined in \eqref{eq:def-kappa}.

\Cref{thm:oracle-inequality} does not provide an explicit choice for the weights $\lambda_1\ge\dots\ge\lambda_p$.
These weights should be large enough so that the event \eqref{eq:event-dual-norm} has probability at least $1/2$.
These weights should also be as small as possible in order to minimize the right hand side of \eqref{eq:oracle-ineq}.
Define $g_1,\dots,g_p$ by
\begin{equation}
\label{eq:def-g_j}
g_j = e_j^\top \Inc^\dagger \eps, \qquad \text{ for all  }j=1,\dots,p \enspace.
\end{equation}
and let $|g|_1^\downarrow\ge\dots\ge |g|_p^\downarrow$ be a nondecreasing rearrangement of $(|g_1|,\dots,|g_p|)$.
Inequality \eqref{ineq:dual-to-max} below reads
\begin{equation}
    (1/n)\norms{\Inc^\dagger \eps }^*
    \le
    \max_{j=1,\dots,p} \left( |g|_j^\downarrow \; / \; (n\lambda_j) \right)\enspace.
\end{equation}
Thus, if the event
\begin{equation}
    \label{eq:event-max}
    \max_{j=1,\dots,p} \left( |g|_j^\downarrow \; / \; (n \lambda_j) \right) \le 1/2\enspace,
\end{equation}
has probability greater than $1/2$,
then the event \eqref{eq:event-dual-norm} has probability greater than $1/2$ as well, and the conclusion of \Cref{thm:oracle-inequality} holds.
This observation can be used to justify the following heuristics for the choice of the tuning parameters $\lambda_1\ge\dots\ge\lambda_p$.
This heuristics can be implemented provided that the Moore-Penrose pseudo-inverse $\Inc^\dagger$
and the probability distribution of the noise random vector $\eps$ are both known.
These heuristics go as follows.
Assume that one has generated $N$ independent copies of the random vector $\eps$,
and denote by $\mathbb{\tilde P}_N$ the empirical probability distribution with respect to these independent copies of $\eps$,
and $\tilde F_N^j$ the empirical cumulative distribution function (cdf) of $|g|_j^\downarrow$.
Next, define $\lambda_j$ as the $(1-1/(3p))$th quantile of $2 |g|_j^\downarrow$, so that
$$\mathbb{\tilde P}_N( 2 |g|_j^\downarrow\le n \lambda_j ) = 1-1/3p\enspace.$$
As $N\to +\infty$, by the Glivenko-Cantelli theorem, $\tilde F_N^j(t)$ converges to the cdf of $|g|_j^\downarrow$ at $t$
uniformly in $t\in\bbR, j\in [p]$.
Hence if $N$ is large enough, then for all $j=1,...,p$ we have
$\mathbb P( 2 |g|_j^\downarrow\le n \lambda_j ) \ge 1-1/2p\enspace.$
By the union bound over $j=1,\dots,p$,
$$
\mathbb{P} \left[ \max_{j=1,\dots,p}  \left(    |g|_j^\downarrow \; / \; (n\lambda_j) \right) \le 1/2 \right] \ge 1/2 \enspace,
$$
thus the event \eqref{eq:event-dual-norm} has probability greater than $1/2$ with respect to the probability distribution $\mathbb P$ of $\eps$.
This simple scheme provides a computational heuristic to choose the weights $\lambda_1,\dots,\lambda_p$.

The following corollaries propose a theoretical choice for the weights.
To state these corollaries, let us write
$$\rho(\G) =  \max_{j\in[p]} \norm{(\Inct )^\dagger e_j}\enspace,$$
following the notation in \cite{hutter2016optimal}.

\begin{cor}\label{cor:rho}
    Assume that the \GSlope weights $\lambda_1\ge\dots\ge\lambda_p\ge0$ satisfy for any $j\in [p]$
    \begin{equation}\label{eq:def_lambdas}
        n \lambda_j \ge 8 \sigma \rho(\G)\sqrt{\log(2p/j)} \enspace.
    \end{equation}
    Then, for any $\delta\in(0,1)$, the oracle inequality \eqref{eq:oracle-ineq}
    holds with probability at least $1-2\delta$.
\end{cor}
Note that if $\lambda_1=\dots=\lambda_p=\lambda$, then the event
\eqref{eq:event-dual-norm} reduces to $\|(\Inct )^\dagger \eps \|_\infty \le n\lambda /2$.
The random variable $\|(\Inct )^\dagger \eps \|_\infty$ is the maximum
of $p$ correlated Gaussian random variables with variance at most $\sigma^2\rho(\G)^2$, so that \eqref{eq:event-dual-norm} has probability at least $1/2$ provided that $\lambda$ is of order $(\rho(\G)\sigma / n)\sqrt{\log p}$.

\begin{proof}[Proof of \Cref{cor:rho}]
    It is enough to show that if the weights $\lambda_1,\dots,\lambda_p$ satisfy \eqref{eq:def_lambdas}
    then the event \eqref{eq:event-dual-norm} has probability at least $1/2$.

    Define $g_1,\dots,g_p$ by \eqref{eq:def-g_j}
    and let $|g|_1^\downarrow\ge\dots\ge |g|_p^\downarrow$ be a nondecreasing rearrangement of $(|g_1|,\dots,|g_p|)$.
    For each $j$, the random variable $g_j$ is a centered Gaussian random variable with variance at most
    $\sigma^2  \norm{\Inc^\dagger e_j}^2 \le \sigma^2 \rho(\mathcal G)^2$.
    By definition of the dual norm, and \cite[Corrolary II.4.3]{Bhatia97}, we have
    \begin{align}
        \frac{1}{n}\norms{\Inc^\dagger \eps }^*
        =&
        \sup_{a\in\bbR^p: \norms{a} = 1} \frac{a^\top \Inc^\dagger \eps}{n}\\
        &
        \le
        \sup_{a\in\bbR^p: \norms{a} = 1}
        \sum_{j=1}^p
        \lambda_j|a|_j^\downarrow \cdot \frac{|g|_j^\downarrow}{n \lambda_j}\\
        &\le
        \max_{j=1,\dots,p} \frac{|g|_j^\downarrow}{n \lambda_j } \label{ineq:dual-to-max}
        \\ & \le \max_{j=1,\dots,p} \frac{|g|_j^\downarrow}{8 \sqrt{V \log(2p/j)} } \enspace,
    \end{align}
    where $V = \sigma^2\max_{j=1,\dots,p}\norm{\Inct^\dagger e_j}^2$.
    Thus, by \Cref{lem:stochastic} below, the event \eqref{eq:event-dual-norm}
    has probability at least $1/2$.
\end{proof}
Under an explicit choice of tuning parameters, \Cref{cor:rho} yields the following result.

\begin{cor}\label{cor:oracle-inequality}
Under the same hypothesis as \Cref{thm:oracle-inequality} but with the special choice $n \lambda_j = 8 \sigma \rho(\G) \sqrt{\log(2p/j)}$ for any $j\in [p]$, then
for any $\delta\in(0,1)$, we have with probability at least $1-2\delta$
    \begin{align}
        \frac 1 n \norm{\xe - \x^\star}^2
        \le \!\!\!
            \inf_{\substack{s\in[p], \x\in\bbR^n\\\norm{\Inct\x}_0\le s}}
            \left[
            \frac 1 n \norm{\x - \x^\star}^2
            \!+\!
            \tfrac{\sigma^2}{n}
                \frac{48 \rho^2(\G) s}{\kappa^2(s)} \log\left(\tfrac{2 e p}{s}\right)
            \right]
            \!+\! \tfrac{\sigma^2}{n}(2 + 16 \log\left(\tfrac{1}{\delta}\right))
        \enspace.
        \label{eq:oracle-ineq_simplified}
    \end{align}
\end{cor}

\begin{proof}
We apply \Cref{lem:control_lambdas} with the choice $C=8 \sigma \rho(\G) / n$
\end{proof}

When the true signal satisfies $\norm{\Inct\x^\star}_0=s^\star$, the previous bound reduces to

    \begin{align}
        \frac 1 n \norm{\xe - \x^\star}^2
        \le
        \frac{\sigma^2}{n}\left(
            \frac{48 \rho(\G)^2 s^\star}{\kappa(s^\star)^2} \log\left(\frac{2 e p}{s^\star}\right) + 2 + 16 \log\left(\frac{1}{\delta}\right)
        \right)
        \enspace.
    \end{align}
\Cref{cor:oracle-inequality} is an improvement w.r.t. the bound provided in \cite[Theorem 2]{hutter2016optimal} for the TV denoiser (also sometimes referred to as the Generalized Lasso) relying on $\ell_1$ regularization defined in Eq.~\eqref{eq:glasso}.

Indeed, the contribution of the second term in \Cref{cor:oracle-inequality}  is reduced from $\log({e p}/{\delta})$ (in \cite[Theorem 2]{hutter2016optimal}) to $\log({2 e p}/{s})$.
Thus the dependence of the right hand side of the oracle inequality in the confidence level $\delta$ is significantly reduced compared to the result of \cite[Theorem 2]{hutter2016optimal}.

A similar bound as in \Cref{cor:oracle-inequality} could be obtained for $\ell_1$ regularization adapting the proof from \cite[Theorem 4.3]{bellec2016slope}. However such a better bound would be obtained for a choice of regularization parameter relying on the $\Inct$-sparsity of the signal. The \GSlope does not rely on such a quantity, and thus \GSlope is adaptive to the unknown $\Inct$-sparsity of the signal.

\begin{rem}
    The optimal theoretical choice of parameter requires the knowledge of the noise level $\sigma$ from the practitioner. Whenever the noise level $\sigma$ is not known, the practitioner can use the corresponding Concomitant estimator to alleviate this issue \cite{Owen07,Belloni_Chernozhukov_Wang11,Sun_Zhang12}, see also \cite{Ndiaye_Fercoq_Gramfort_Leclere_Salmon16} for efficient algorithms to compute such scale-free estimators.
\end{rem}


\section{Numerical experiments}
\label{sec:experiments}

\subsection{Algorithm for \GSlope}

In this section, we propose an algorithm to compute a solution of the highly structured optimization problem \eqref{eq:gslope}.
The data fidelity term $f:\beta \mapsto \norm{y - \x}_2^2/2$ is a convex smooth function with $1$-Lipschitz gradient, and the map $\x \mapsto \norms{\Inct \x}$ is the pre-composition by a linear operator of the norm $\norms{\cdot}$ whose proximal operator can be easily computed~\cite{Zeng_Figueiredo14,
bogdan2015slope}.
Thus, the use of a dual or primal-dual proximal scheme can be advocated.

Problem~\eqref{eq:gslope} can be rewritten as
\begin{equation}
  \umin{\x \in \RR^n} f(\x) + g(\Inct \x) \enspace,
\end{equation}
where $f$ is a smooth, $1$-Lipschitz strictly convex function and $g=\norms{\cdot}$ is a convex, proper, lower semicontinuous function (see for instance \cite[p.~275]{Bauschke_Combettes11}).
Its dual problem reads
\begin{equation}
  \umin{\theta \in \RR^p} f^\star(\Inc \theta) + g^\star(-\theta) \enspace,
\end{equation}
where $f^\star$ is the convex conjugate of $f$, \ie for any $x\in\bbR^n$
\begin{equation}
  f^\star(x) = \sup_z \dotp{x}{z} - f(z) \enspace.
\end{equation}
Classical computations leads to the following dual problem
\begin{equation}\label{eq:gslope-dual}
  \umin{\theta \in \RR^p} \frac{1}{2} \norm{\Inc \theta - y}_2^2 - \frac{1}{2} \norm{y}_2^2
  \qsubjq
  \norms{\theta}^* \leq 1 \enspace.
\end{equation}
The dual formulation~\eqref{eq:gslope-dual} can be rewritten as an unconstrained problem, using for any set $\mathcal{C}\subset\bbR^n$, and any $\theta\in\bbR^n$, the notation
$$\iota_{\mathcal{C}}(\theta)=\begin{cases}
0, &\text{ if } \theta \in \mathcal{C}\\
+\infty, &\text{ otherwise}
\end{cases}.
$$

The quadratic term in $y$ is constant and can be dropped.
Thus the optimization problem \eqref{eq:gslope-dual} is equivalent to
\begin{equation}\label{eq:gslope-dual-unc}
  \umin{\theta \in \RR^p} \frac{1}{2} \norm{\Inc \theta - y}_2^2 + \iota_{\{\norms{\cdot}^* \leq 1\}}(\theta) .
\end{equation}
The formulation in ~\eqref{eq:gslope-dual-unc} is now well suited to apply an accelerated version of the forward-backward algorithm such as FISTA~\cite{Beck_Teboulle09}.
As a stopping criterion, we use a duality gap criterion: $\Delta(\x, \theta)\leq \epsilon$, where
\begin{equation}
  \Delta(\x, \theta) = \frac{1}{2}\norm{y - \x}_2^2 + \norms{\Inct \x} + \frac{1}{2} \norm{\Inc \theta - y}_2^2 - \frac{1}{2} \norm{y}_2^2 \enspace,
\end{equation}
 for a feasible pair $(\x, \theta)$ and by
$\Delta(\x, \theta) = +\infty$ for an unfeasible pair. In practice we set $\epsilon=10^{-2}$ as a default value.
\Cref{alg:fista} summarizes the dual FISTA algorithm applied to the \GSlope minimization problem.

\begin{algorithm}[t]
  \caption{FISTA on dual formulation}\label{alg:fista}
  \begin{algorithmic}
    \Require $(\x^0, \theta^0)$ initial guess, $L = \norm{\Inc}^2$, $t_0 = 1$, $\epsilon$ duality gap tolerance
    \State $k\gets0$
    \While{$\Delta(\x^k, \theta^k) > \epsilon$}
    \State $\theta^{k+1} \gets \Proj_{\frac{1}{L}B_*}  \left(\bar \theta^{k} - \frac{1}{L} (\Inct(\Inc \bar \theta^{k} - y))\right)$ \Comment{forward-backward step}
    \State $\x^{k+1} \gets y - \Inc \theta^{k+1}$ \Comment{necessary to compute $\Delta$}
    \State $t_{k+1} \gets \frac{1 + \sqrt{1 + 4t_k^2}}{2}$ \Comment{FISTA rule}
    \State $\bar \theta^{k+1} \gets \theta^{k+1} + \frac{t_k-1}{t_{k+1}} (\theta^{k+1} - \theta^k)$ \Comment{non-convex over-relaxation}
    \State $k \gets k + 1$
    \EndWhile
  \end{algorithmic}
\end{algorithm}

We recall that the proximity operator of a convex, proper, lower semicontinuous function $f$ is given as the unique solution of the optimization problem
\begin{equation}
  \Prox_{\lambda f}(\x) = \uargmin{z \in \RR^n} \frac{1}{2} \norm{\x - z}_2^2 + \lambda f(z) \enspace.
\end{equation}
To compute the proximity operator of $\iota_{\{\norms{\cdot}^* \leq 1\}}$, we use the Moreau's decomposition \cite[p.~65]{Parikh_Boyd_Chu_Peleato_Eckstein13} which links it to the proximity operator of the dual \Slope-norm,
\begin{align}
  \theta
  &= \Prox_{\tau \iota_{\norms{\cdot}^* \leq 1}}(\theta) + \tau \Prox_{\frac{1}{\tau} \norms{\cdot}} \left(\frac{\theta}{\tau}\right) \\
  &= \Pi_{\frac{1}{\tau} B_*}(\theta) + \tau \Prox_{\frac{1}{\tau} \norms{\cdot}} \left(\frac{\theta}{\tau}\right) \enspace,
\end{align}
where $\Pi_{\frac{1}{\tau} B_*}$ is the projection onto the unit ball $B_*$ associated to the dual norm $\norm{\cdot}^*$ scaled by a factor $1/\tau$.
The proximity operator of $\norms{\cdot}$ can be obtained obtained in several ways \cite{Zeng_Figueiredo14,bogdan2015slope}. In our numerical experiments, we use the connection between this operator and the isotonic regression following \cite{Zeng_Figueiredo14}, which can be computed in linear time.
Under the assumption that the quantity $(u_i - \lambda_i)$ is positive, non-increasing (which is obtained by sorting $|u|$ and restoring the signs and ordering afterwards, see details in \cite[Eq. (24)]{Zeng_Figueiredo14}), computing $\Prox_{\norms{\cdot}}(u)$ is equivalent to solving the problem
\begin{equation}
  \uargmin{\theta \in \RR^p} \frac{1}{2} \norm{u - \lambda - \theta}_2^2
  \qsubjq
  \theta_1 \geq \theta_2 \geq \cdots \geq \theta_p \geq 0 \enspace.
\end{equation}
We have relied on the fast implementation implementation of the PAVA algorithm (for Pool Adjacent Violators Algorithm, see for instance \cite{best1990active}), available in \texttt{scikit-learn} \cite{Pedregosa_etal11} to solve this inner problem.

The source code used for our numerical experiments is freely available at \url{http://github.com/svaiter/gslope_oracle_inequality}.

\subsection{Synthetic experiments}

To illustrate the behavior of \GSlope, we first propose two synthetic experiments in moderate dimension.
The first one is concerned with the so-called ``Caveman'' graph and the second one with the 1D path graph.

For these two scenarios, we analyze the performance following the same protocol.
For a given noise level $\sigma$, we use the bounds derived in~\Cref{thm:oracle-inequality} (we dropped the constant term 8) and in~\cite{hutter2016optimal}, \ie
\begin{equation}\label{eq:xp-practical-bounds}
  \lambda_{\rm GL} = \rho(\Gg) \sigma \sqrt{\frac{2 \log(p)}{n}}
  \qandq
  (\lambda_{\rm GS})_j = \rho(\Gg) \sigma \sqrt{\frac{2\log(p/j)}{n}} \quad \forall j \in [p] \enspace.
\end{equation}
For every $n_0$ between 0 and $p$, we generate 1000 signals as follows.
We draw $J$ uniformly at random among all the subsets of $\nint{p}$ of size $n_0$.
Then, we let $\Pi_J$ be the projection onto $\Ker \Inc_J^\top$ and generate a vector $g \sim \Nn(0,\Id_n)$.
We then construct $\xs = c (\Id - \Pi_J) g$ where $c$ is a given constant (here $c=8$).
This constrains the signal $\xs$ to be of $\Inct$-sparsity at most $p-n_0$.

We corrupt the signals by adding a zero mean Gaussian noise with variance $\sigma^2$, and run both the Graph-Lasso estimator and the Graph-Slope estimator. We then compute the mean of the mean-squared error (MSE), the false detection rate (FDR)
and the true detection rate (TDR).
To clarify our vocabulary, given an estimator $\xe$ and a ground truth $\xs$, the MSE reads $(1/n)\norm{\xs - \xe}^2$, while the FDR and TDR read, respectively,
\begin{equation}
 {\rm FDR}(\xe,\xs) =
 \begin{cases}
   \frac{\abs{\enscond{j \in [p]}{j \in \supp(\Inct \xe) \text{ and } j \not\in \supp(\Inct \xs) }}}{\abs{\supp(\Inct \xe)}} & \text{if } \Inct \xe \neq 0 \\
   0 & \text{if } \Inct \xe = 0 ,
 \end{cases}
\end{equation}
and
\begin{equation}
  {\rm TDR}(\xe,\xs) =
  \begin{cases}
    \frac{\abs{\enscond{j \in [p]}{j \in \supp(\Inct \xe) \text{ and } j \in \supp(\Inct \xs) }}}{\abs{\supp(\Inct \xs)}}, & \text{if } \Inct \xs \neq 0, \\
    0, & \text{if } \Inct \xs = 0 ,
  \end{cases}
\end{equation}
where for any $z\in \bbR^p$, $\supp(z) =\enscond{j \in [p]}{ z_j \neq 0}$.

\paragraph{Example on Caveman}

The caveman model was introduced in~\cite{watts1999networks} to model small-world phenomenon in sociology.
Here we consider its relaxed version, which is a graph formed by $l$ cliques of size $k$ (hence $n=lk$), such that with probability $q \in [0,1]$, an edge of a clique is linked to a different clique.
In our experiment, we set $l=4$, $k=10$ ($n=40$) and $q=0.1$.
We provide a visualisation of such a graph in Figure~\ref{fig:caveman-graph}.
For this realization, we have $p=180$. 
The rewired edges are indicated in blue in Figure~\ref{fig:caveman-graph} whereas the edges similar to the complete graph on 10 nodes are in black.
The signals are generated as random vectors of given $\Inct$-sparsity with a noise level of $\sigma=0.2$.
Figure~\ref{fig:caveman-weights} shows the weights decay.

Figures~\ref{fig:caveman-mse}--\ref{fig:caveman-tdr} represent the evolution of the MSE and TDR in function of the level of $\Inct$-sparsity.
We observe that while the MSE is close between the \GLasso and the \GSlope estimator at low level of sparsity, the TDR is vastly improved in the case of \GSlope, with a small price concerning the FDR (a bit more for the Monte Carlo choice of the weights).
Hence empirically, \GSlope will make more discoveries than \GLasso without impacting the overall FDR/MSE, and even improving it.

\begin{figure}[t]
  \centering
  \begin{subfigure}[b]{0.4\linewidth}
    \centering\includegraphics[trim={3cm 1cm 3cm 1cm},clip,width=\textwidth]{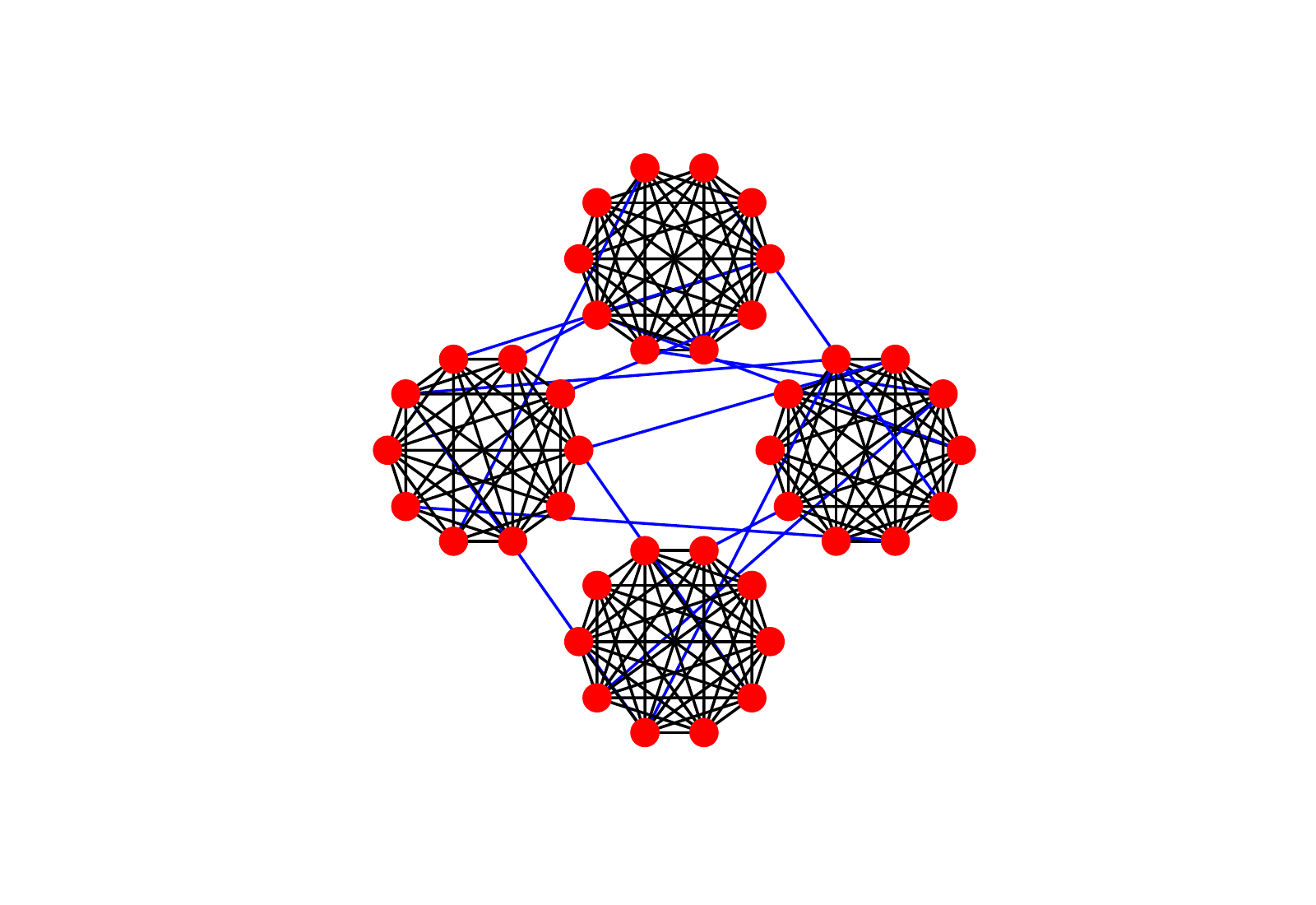}
    \caption{\label{fig:caveman-graph} Realization of a caveman graph}
  \end{subfigure}%
  \begin{subfigure}[b]{0.4\linewidth}
    \centering\includegraphics[width=\textwidth]{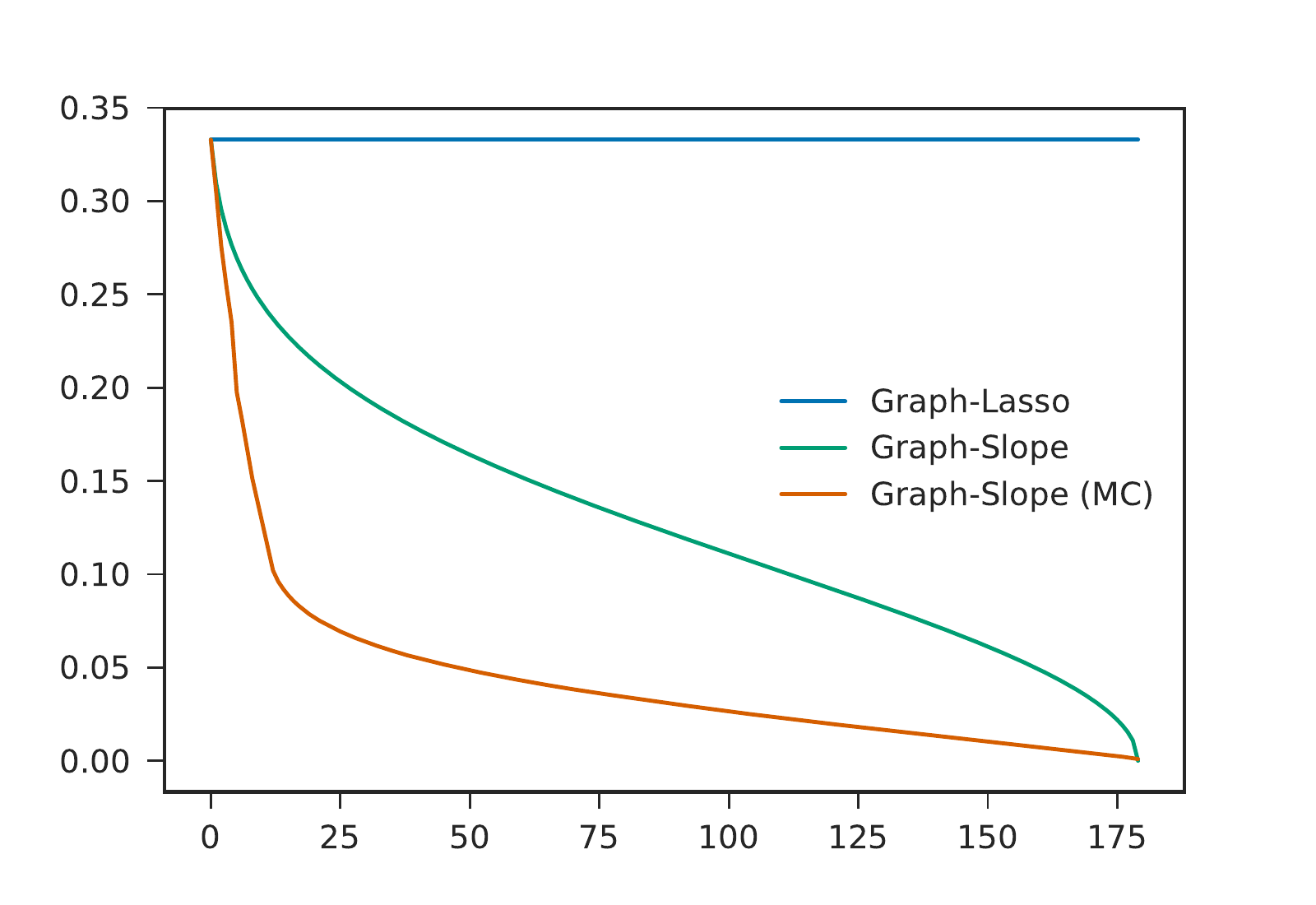}
    \caption{\label{fig:caveman-weights} Weights}
  \end{subfigure}
  \begin{subfigure}[b]{0.3\linewidth}
    \centering\includegraphics[width=\textwidth]{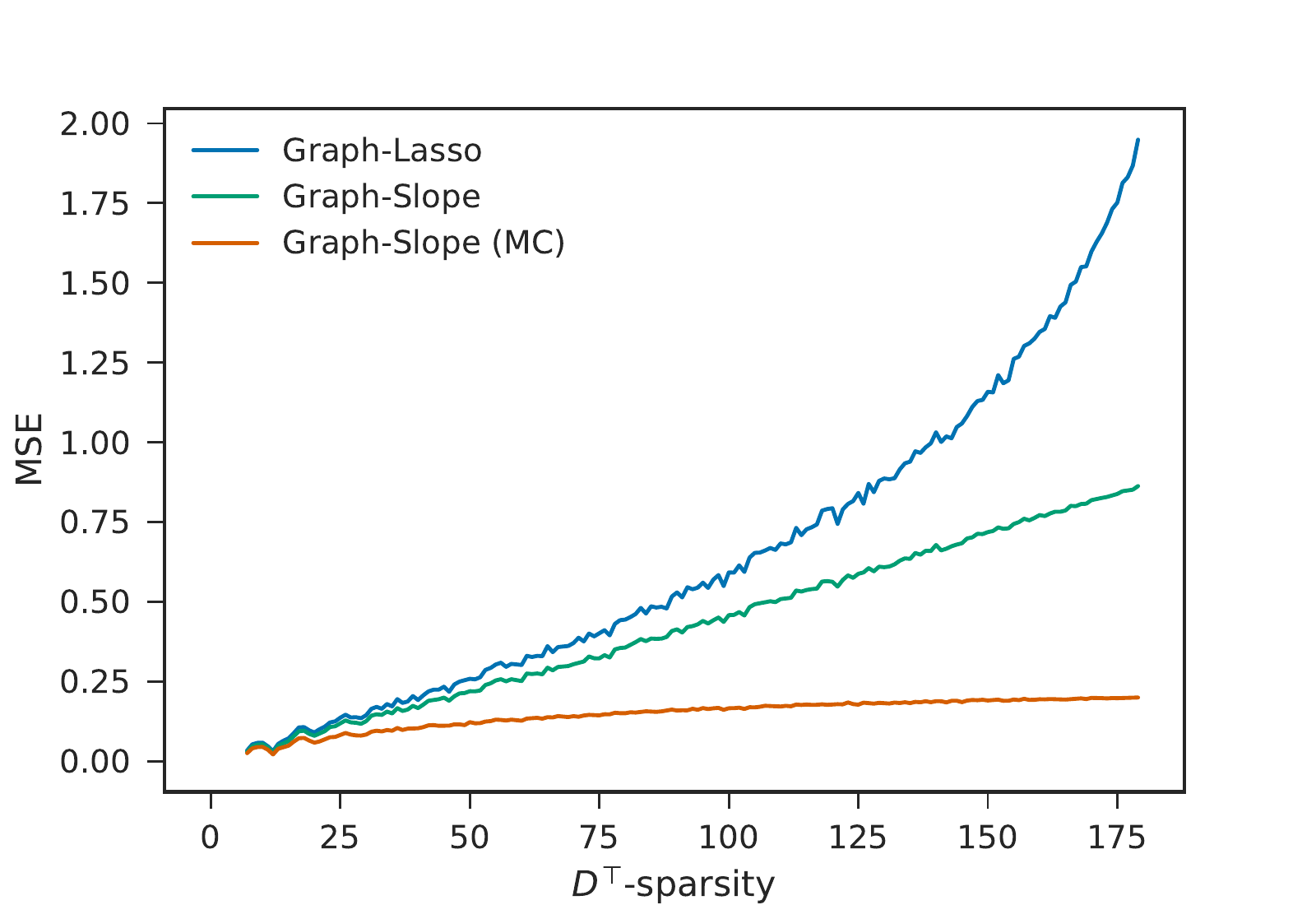}
    \caption{\label{fig:caveman-mse} Mean-square error\\ (MSE)}
  \end{subfigure}
  \begin{subfigure}[b]{0.3\linewidth}
    \centering\includegraphics[width=\textwidth]{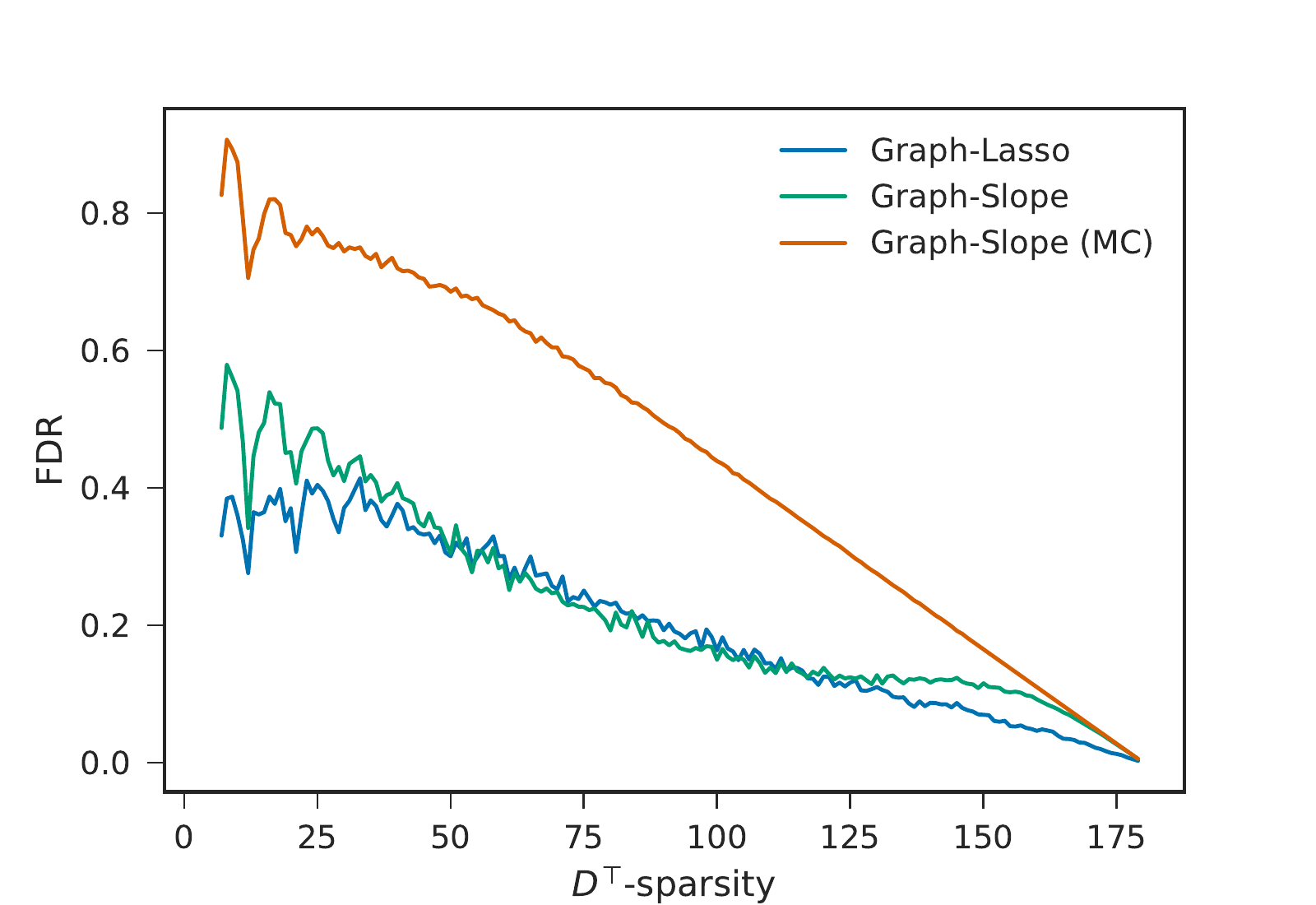}
    \caption{\label{fig:caveman-fdr} False Detection Rate\\ (FDR)}
  \end{subfigure}%
  \begin{subfigure}[b]{0.3\linewidth}
    \centering\includegraphics[width=\textwidth]{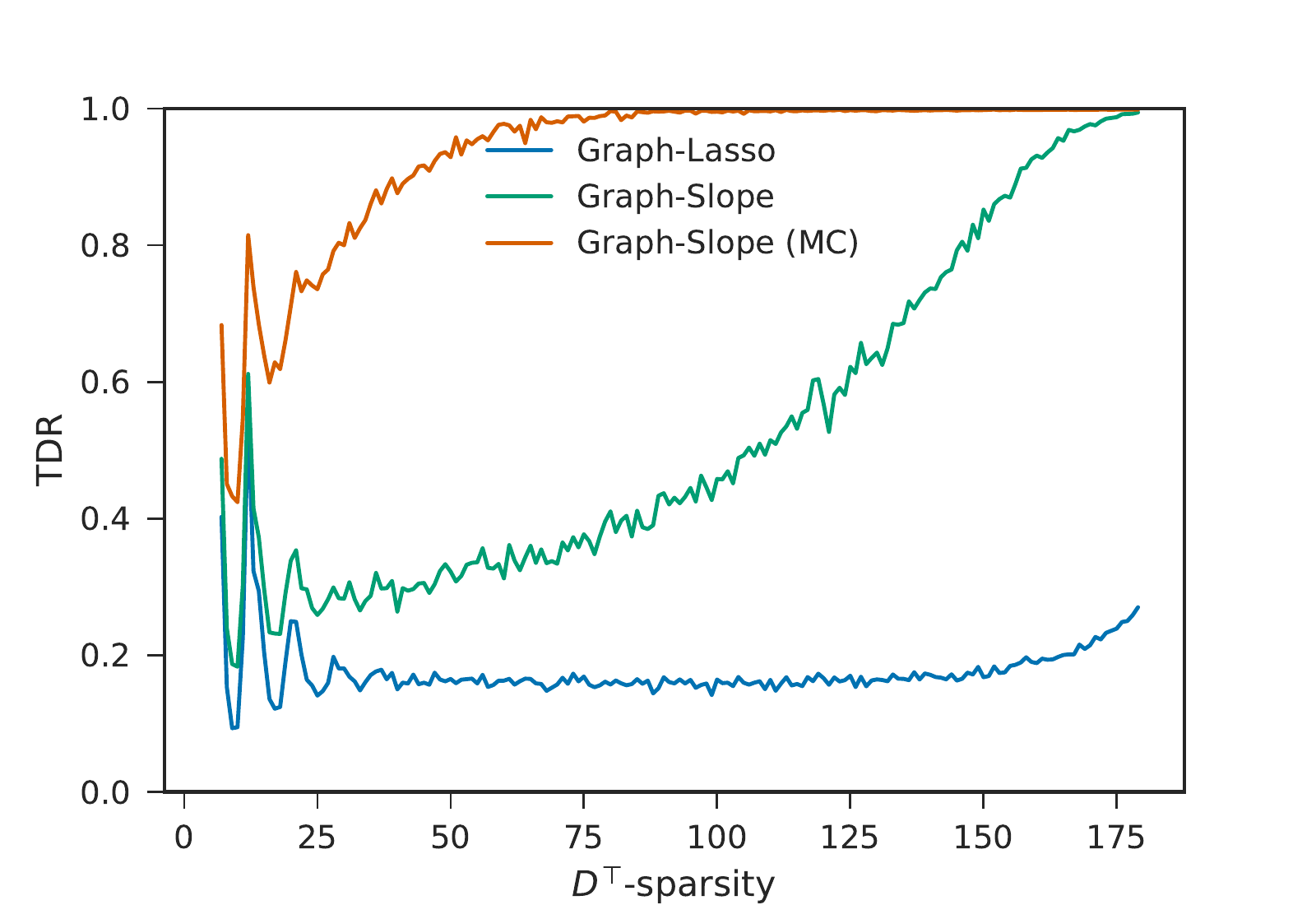}
    \caption{\label{fig:caveman-tdr} True Detection Rate\\ (TDR)}
  \end{subfigure}
  \caption{Relaxed caveman denoising}
\end{figure}

\paragraph{Example on a path: 1D--Total Variation}

The classical 1D--Total Variation corresponds to the \GLasso estimator $\hat{\beta}^{\rm GL}$ when $\G$ is the path graph over $n$ vertices, hence with $p=n-1$ edges.
In our experiments, we take $n=100$, $\sigma = 0.6$ and a very sparse gradient ($s=4$). According to these values, and taking a random amplitude for each step, we generate a piecewise-constant signal.
We display a typical realization of such a signal in Figure~\ref{fig:tv1d-ex}.
Figure~\ref{fig:tv1d-weights} shows the weights decay. Note that in this case, the Monte--Carlo weights shape differs from the one in the previous experiment.
Indeed, they are adapated to the underlying graph, contrary to the theoretical weights $\lambda_{\rm GS}$ which depend only on the size of the graph.
Figures~\ref{fig:tv1d-mse}--\ref{fig:tv1d-tdr} represent the evolution of the MSE and TDR in function of the level of $\Inct$-sparsity.
Here, \GSlope does not improve the MSE significantly.
However, as for the caveman experiments, \GSlope is more likely to make more discoveries than \GLasso for a small price concerning the FDR.

\begin{figure}[t]
  \centering
  \begin{subfigure}[b]{0.4\linewidth}
    \centering\includegraphics[width=\textwidth]{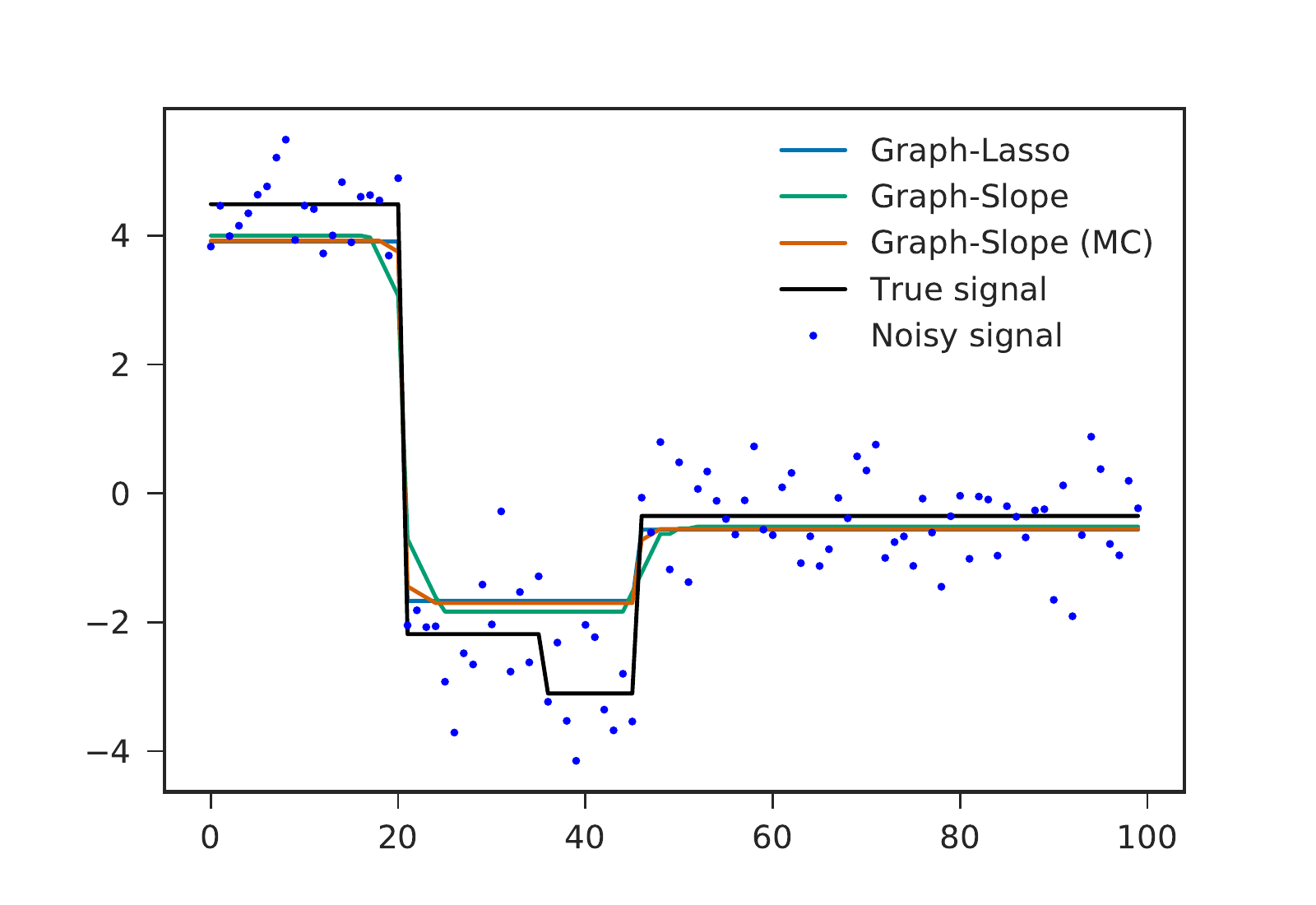}
    \caption{\label{fig:tv1d-ex} Example of signal}
  \end{subfigure}%
  \begin{subfigure}[b]{0.4\linewidth}
    \centering\includegraphics[width=\textwidth]{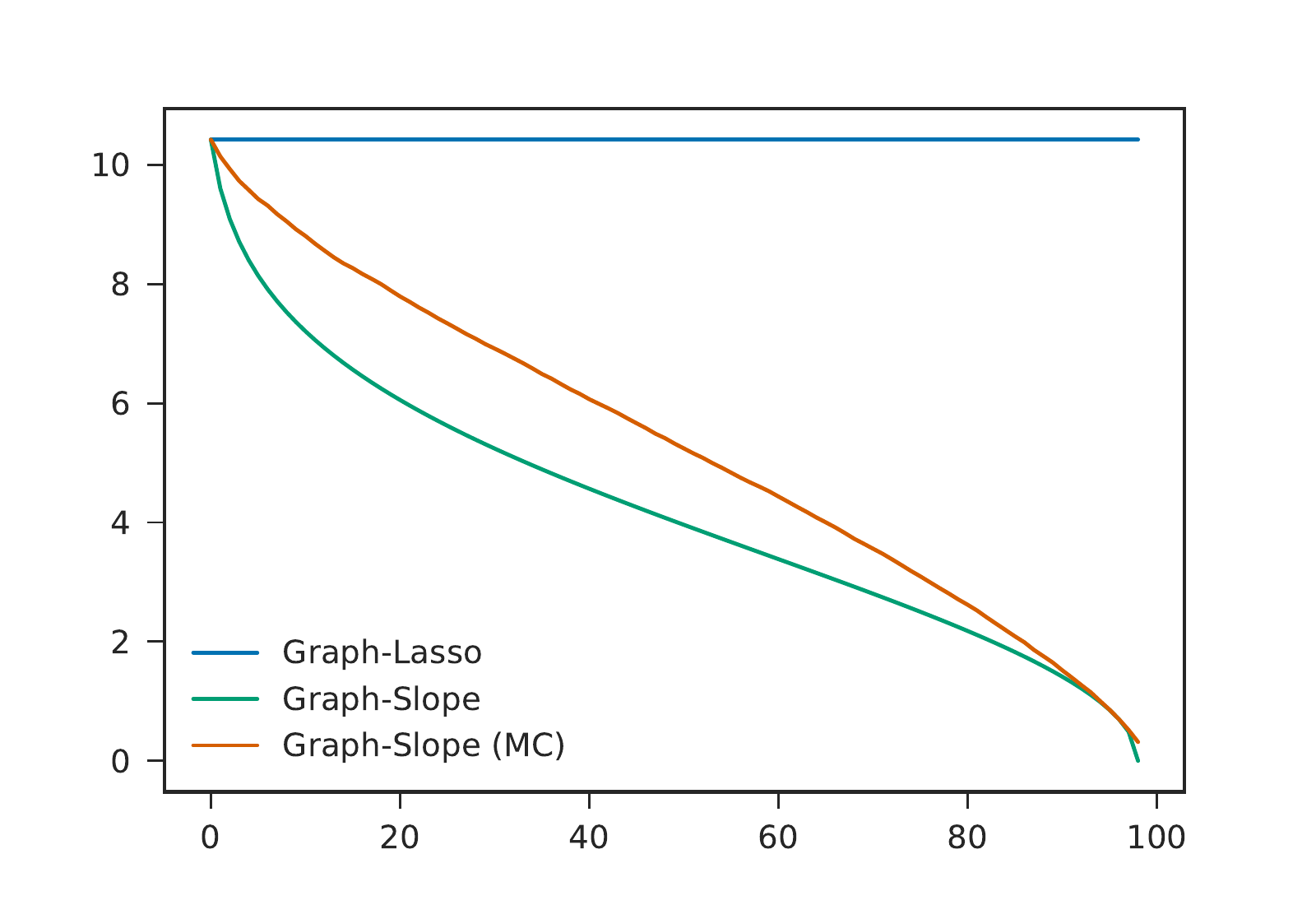}
    \caption{\label{fig:tv1d-weights} Weights}
  \end{subfigure}\\
  \begin{subfigure}[b]{0.3\linewidth}
    \centering\includegraphics[width=\textwidth]{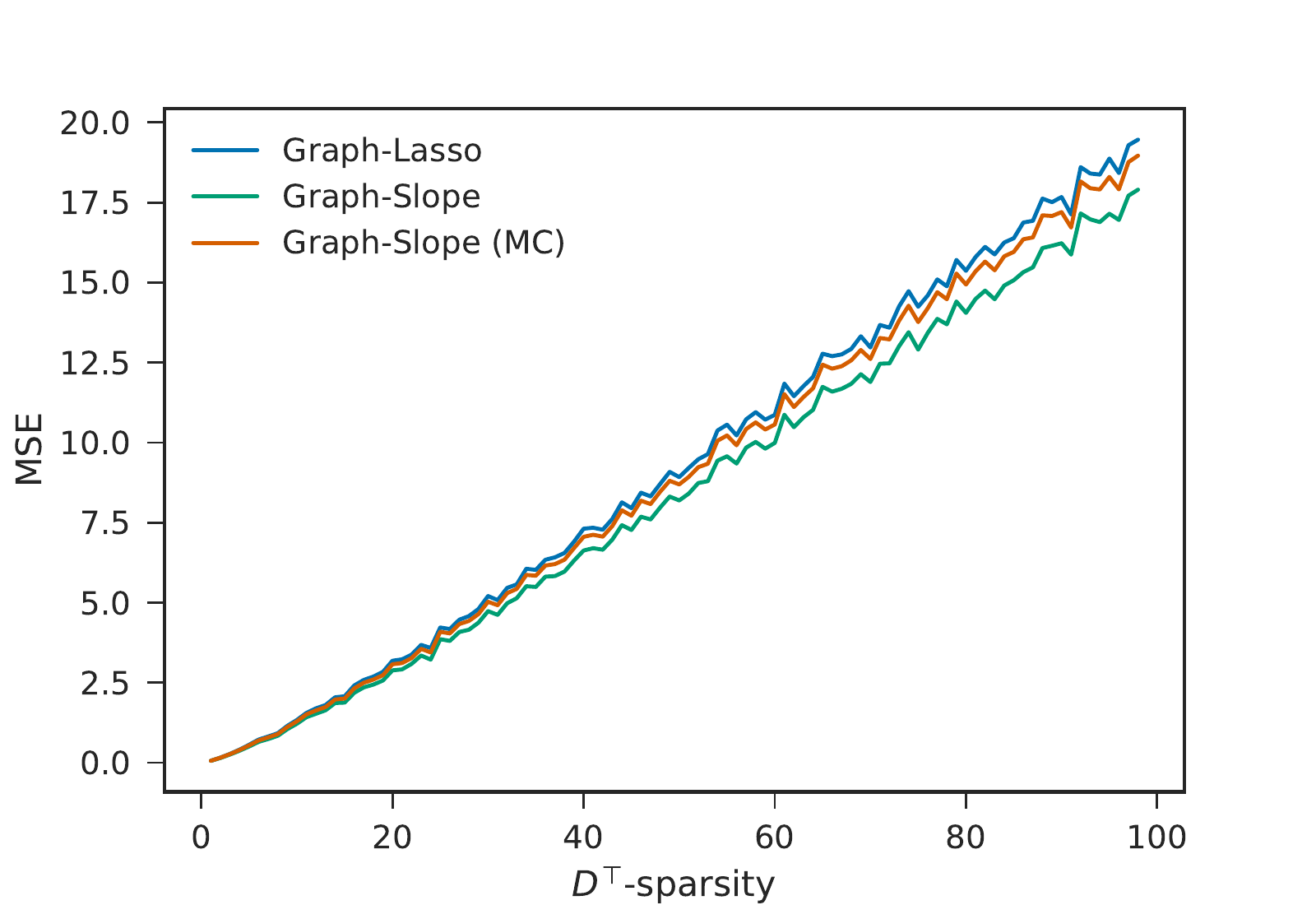}
    \caption{\label{fig:tv1d-mse} Mean-square error\\ {\phantom{\quad\quad\quad\quad} (MSE)}}
  \end{subfigure}
  \begin{subfigure}[b]{0.3\linewidth}
    \centering\includegraphics[width=\textwidth]{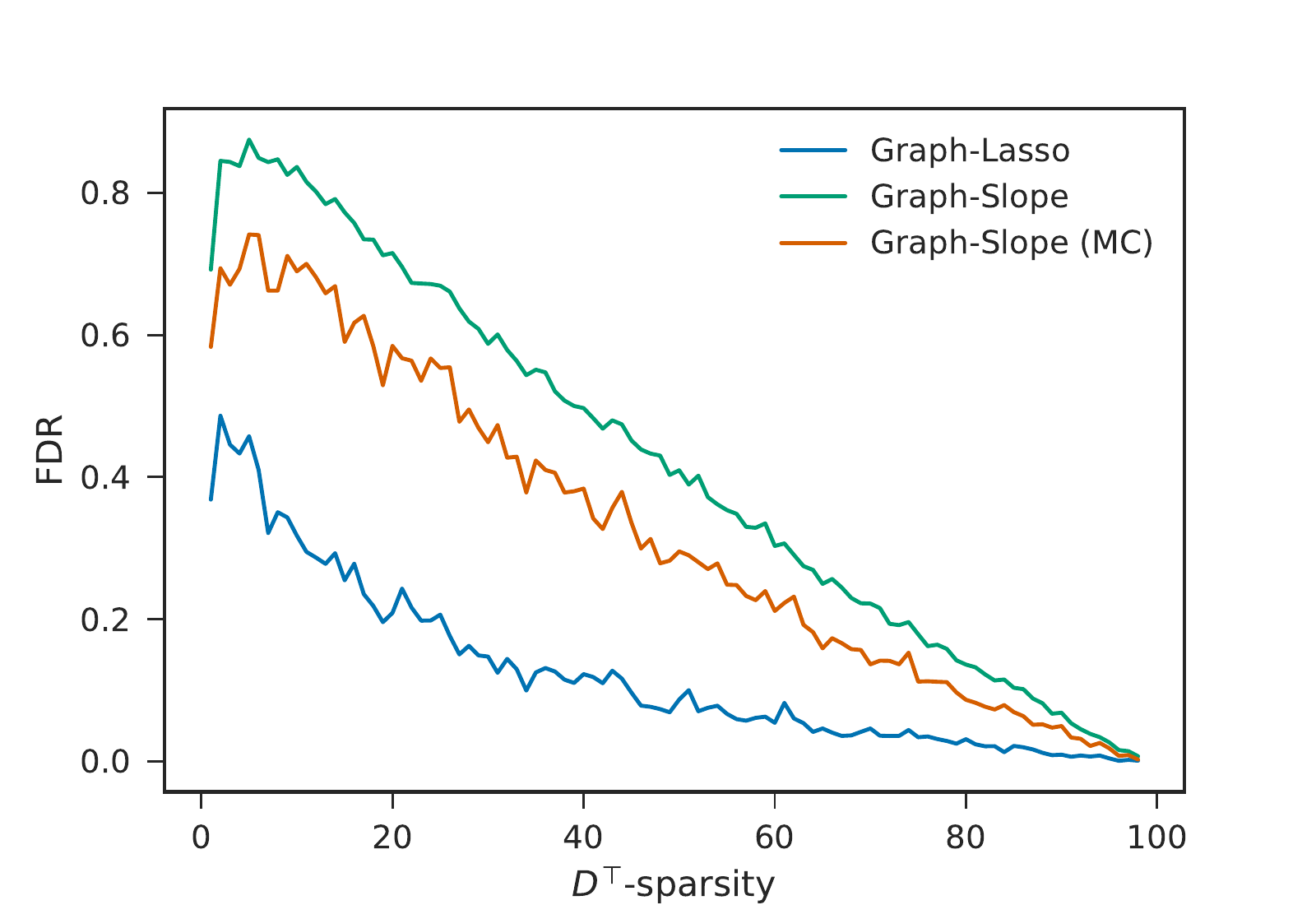}
    \caption{\label{fig:tv1d-fdr} False Detection Rate\\ {\phantom{\quad\quad\quad\quad} (FDR)}}
  \end{subfigure}%
  \begin{subfigure}[b]{0.3\linewidth}
    \centering\includegraphics[width=\textwidth]{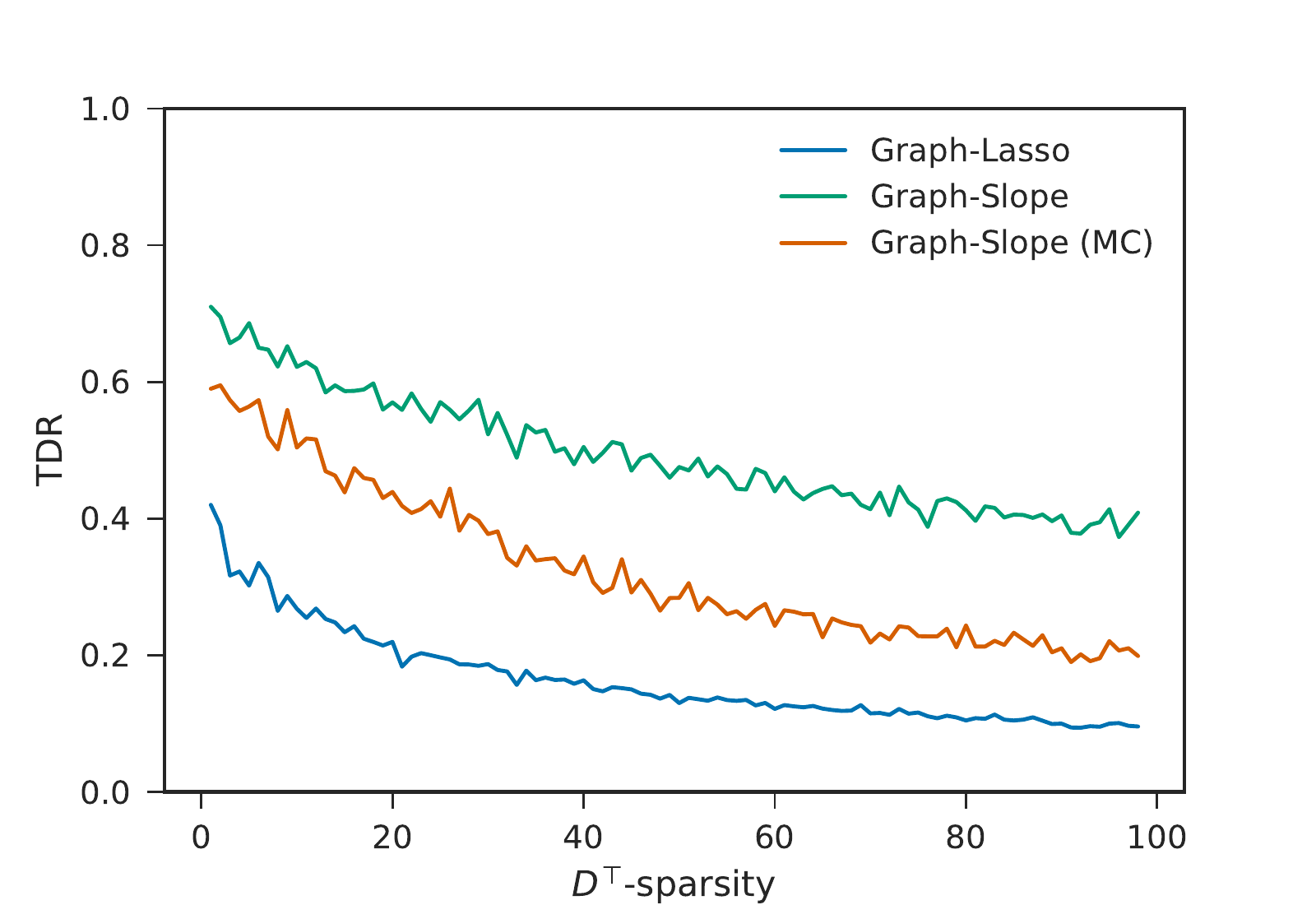}
    \caption{\label{fig:tv1d-tdr} True Detection Rate\\ {\phantom{\quad\quad\quad\quad} (TDR)}}
  \end{subfigure}
  \caption{TV1D}
\end{figure}

\subsection{Example on real data: Paris roads network}

To conclude our numerical experiments, we present our result on a real-life graph, the road network of Paris, France.
Thanks to the Python module \texttt{osmnx}~\cite{boeing2016osmnx}, which downloads and simplifies OpenStreetMap data, we run our experiments on $p=20108$ streets (edges) and $n=10205$ intersections (vertices).

The ground truth signal is constructed as in~\cite{fan2017pw} as follows. Starting from 30 infection sources, each infected intersections has probability $0.75$ to infect each of its neighbors. We let the infection process run for 8 iterations.
The resulting graph signal $\xs$ is represented in Figure~\ref{fig:paris-x0} with $\Inct$-sparsity 1586.
We then corrupt this signal by a zero mean Gaussian noise with standard-deviation $\sigma=0.8$, leading to the observations $y$ represented in Figure~\ref{fig:paris-y}.

Instead of using the parameters given in~\eqref{eq:xp-practical-bounds}, we have computed the oracle parameters for the \GLasso and \GSlope estimators by evaluating for 100 parameters of the form
\begin{equation}\label{eq:xp-practical-bounds-2}
  \lambda_{\rm GL} = \alpha \sigma \sqrt{\frac{2 \log(p)}{n}}
  \qandq
  (\lambda_{\rm GS})_j = \alpha \sigma \sqrt{\frac{2\log(p/j)}{n}} \quad \forall j \in [p] \enspace ,
\end{equation}
where $\alpha$ lives on a geometric grid inside $[10^{-5},10^{1.5}]$.
The best one in terms of MSE (\ie in term of $(1/n)\norm{\xe - \xs}^2$) is refered to as the oracle parameter.
The results are illustrated in Figure~\ref{fig:paris-lasso} for \GLasso and in Figure~\ref{fig:paris-slope} for \GSlope.
We can see the benefit of \GSlope, for instance in the center of Paris where the sources of infections are better identified as shown in the close-up, see Figures~\ref{fig:paris-x0-close}--\ref{fig:paris-slope-close}.

\begin{figure}[t]
  \centering
  \begin{subfigure}[b]{0.47\linewidth}
    \centering\includegraphics[width=\textwidth]{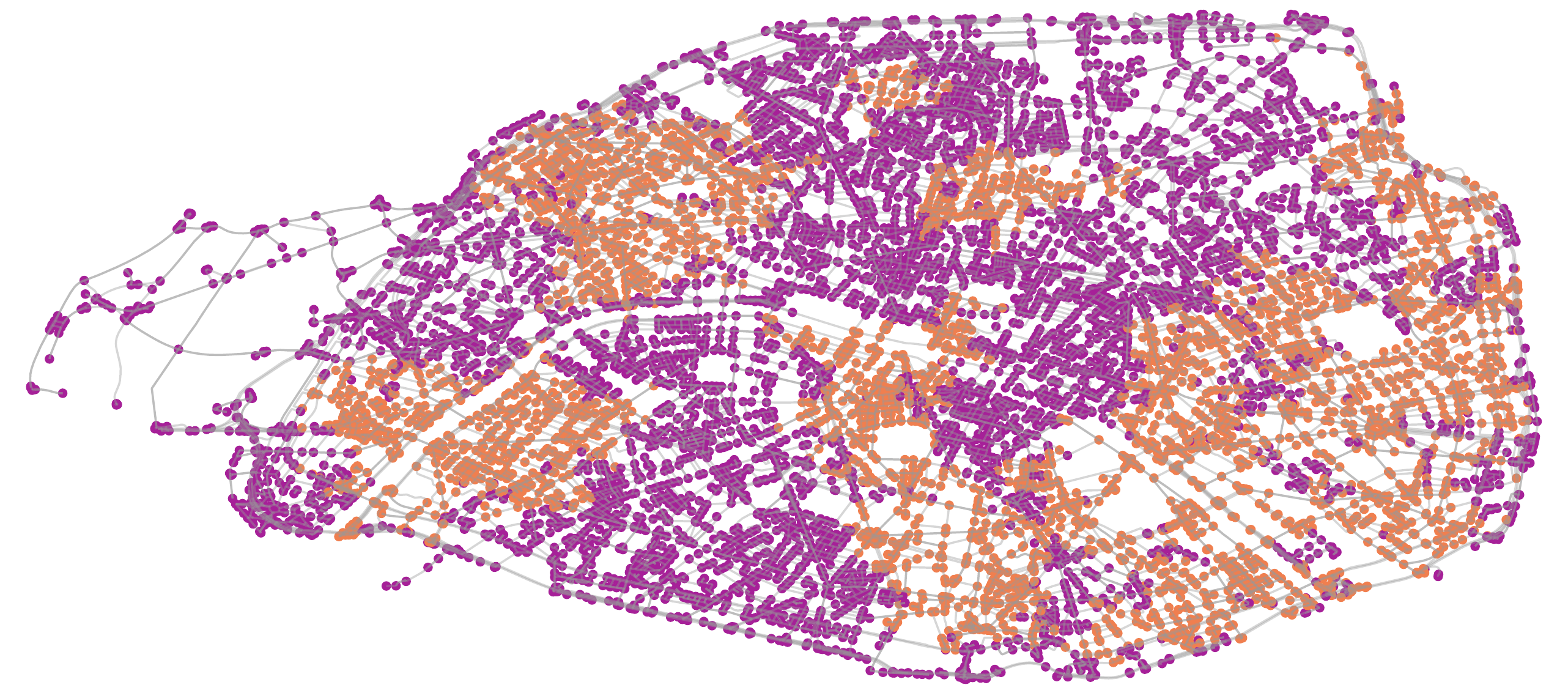}
    \caption{\label{fig:paris-x0} True signal $\xs$}
  \end{subfigure}%
  \begin{subfigure}[b]{0.47\linewidth}
    \centering\includegraphics[width=\textwidth]{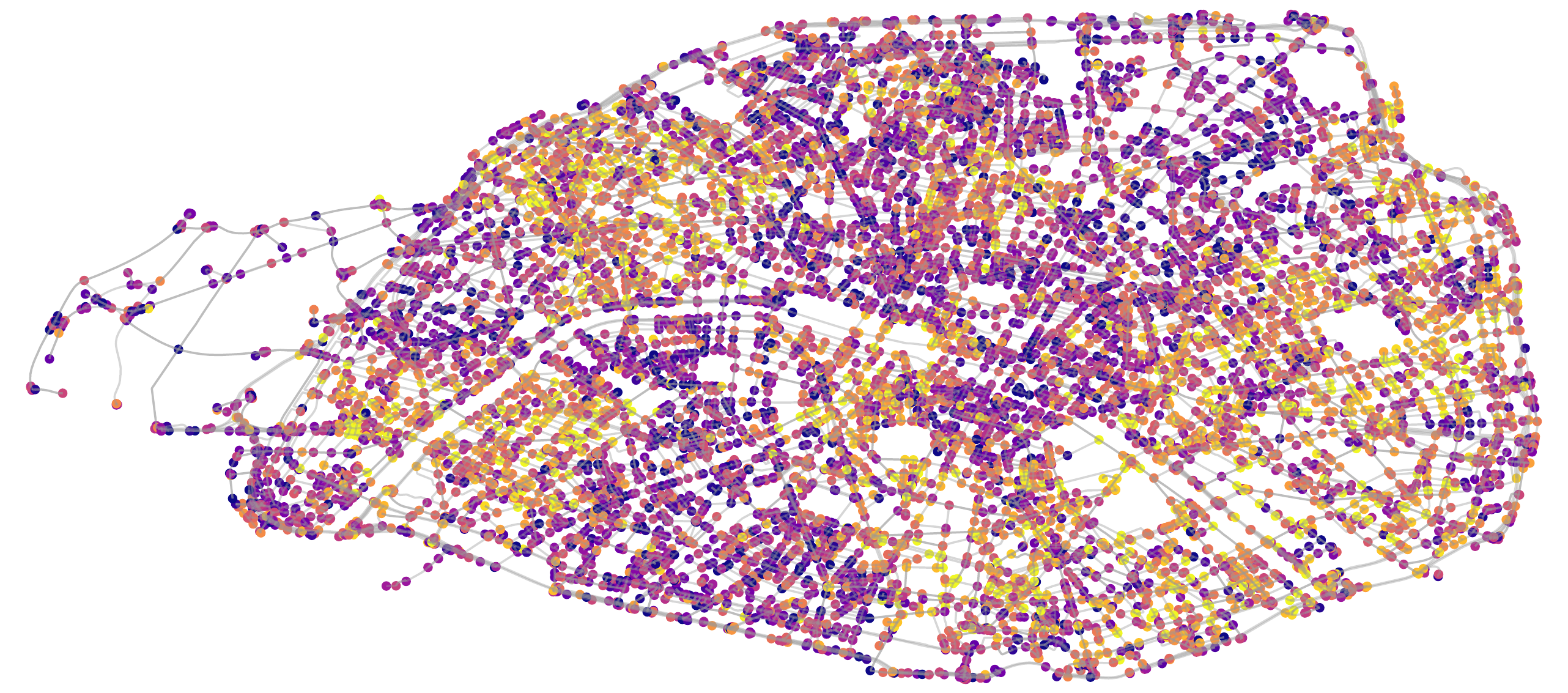}
    \caption{\label{fig:paris-y} Noisy signal $y$}
  \end{subfigure}\\
  \begin{subfigure}[b]{0.47\linewidth}
    \centering\includegraphics[width=\textwidth]{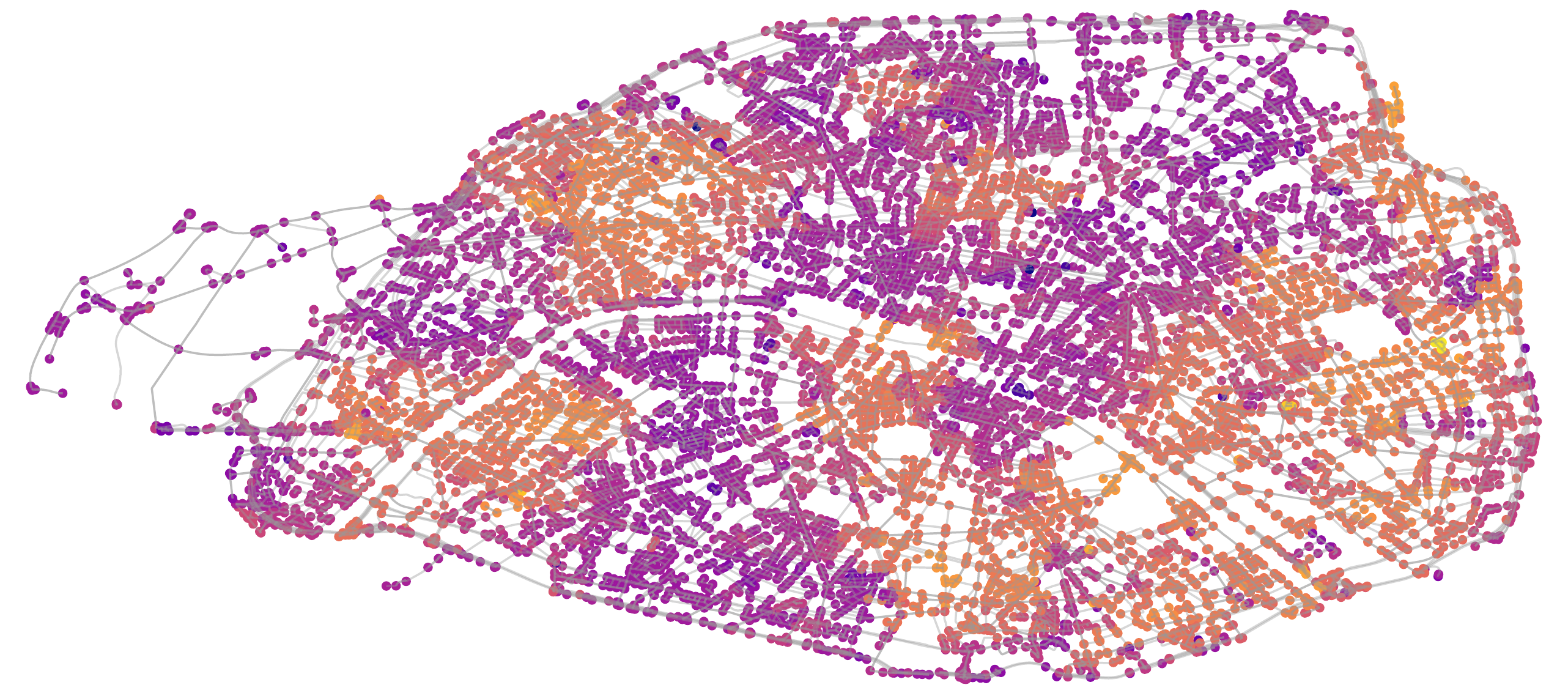}
    \caption{\label{fig:paris-lasso} \GLasso $\hat{\beta}^{\rm GL}$\\MSE=0.070,FDR=83.4\%,\\ TDR=52.1\%}
  \end{subfigure}%
  \begin{subfigure}[b]{0.47\linewidth}
    \centering\includegraphics[width=\textwidth]{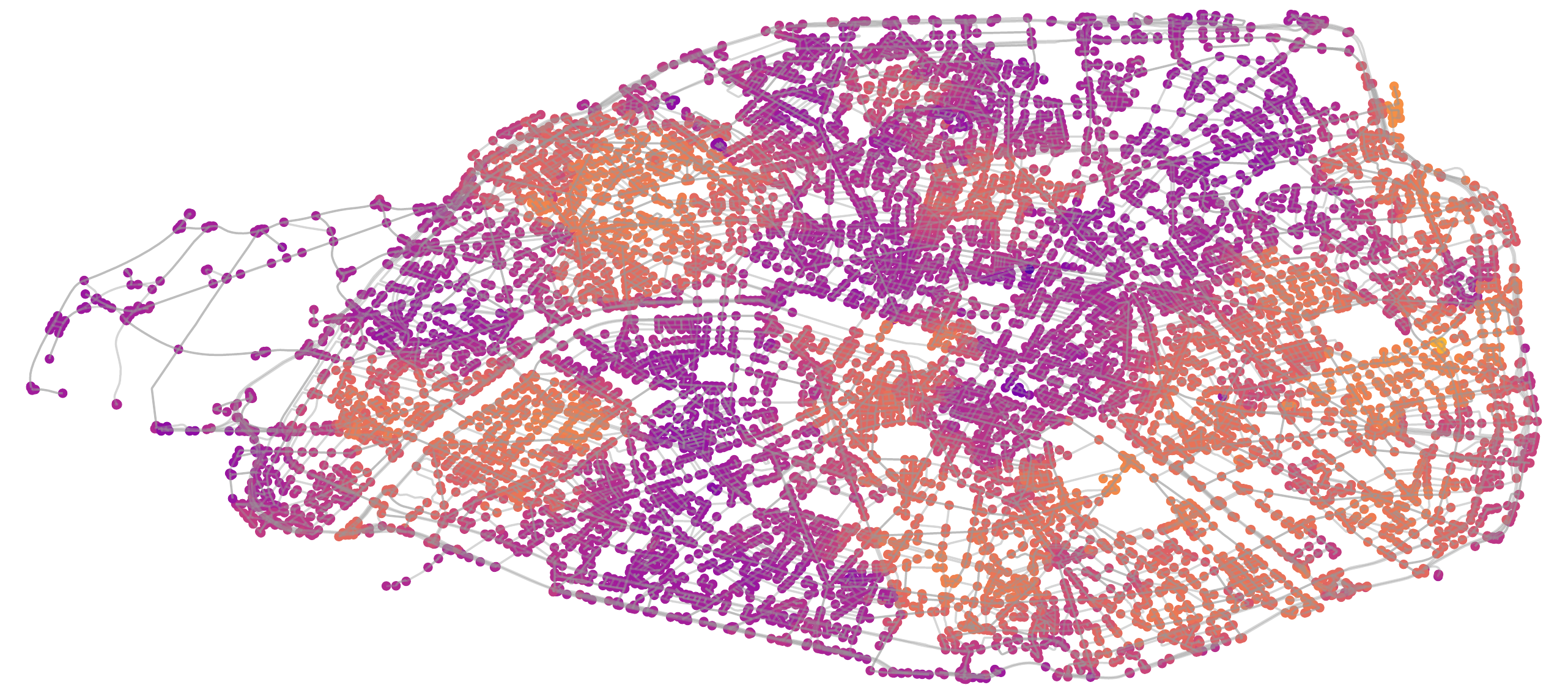}
    \caption{\label{fig:paris-slope} \GSlope $\hat{\beta}^{\rm GS}$\\
    MSE=0.074,FDR=88.6\%,\\ TDR=73.5\%}
  \end{subfigure}\\
  \begin{subfigure}[b]{0.32\linewidth}
    \centering\includegraphics[trim={13cm 4cm 7cm 4cm},clip,width=\textwidth]{paris-x0}
    \caption{\label{fig:paris-x0-close} True signal close-up}
  \end{subfigure}\hspace{0.1cm}
  \begin{subfigure}[b]{0.32\linewidth}
    \centering\includegraphics[trim={13cm 4cm 7cm 4cm},clip,width=\textwidth]{paris-lasso}
    \caption{\label{fig:paris-lasso-close} \GLasso close-up}
  \end{subfigure}\hspace{0.1cm}
  \begin{subfigure}[b]{0.32\linewidth}
    \centering\includegraphics[trim={13cm 4cm 7cm 4cm},clip,width=\textwidth]{paris-slope}
    \caption{\label{fig:paris-slope-close} \GSlope close-up}
  \end{subfigure}\\
  \centering\includegraphics[trim={0cm 1cm 0cm 8cm},clip,width=\textwidth]{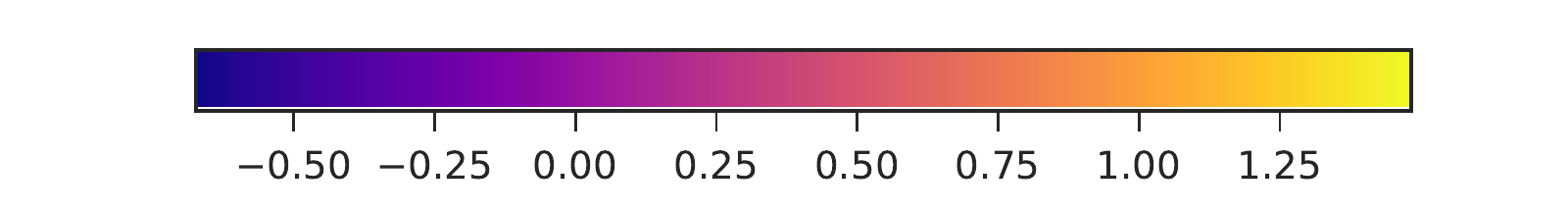}
  \caption{Paris road network. Comparison of oracle choice for the tuning parameter between \GLasso and \GSlope.}
\end{figure}

\appendix



\section{Preliminary lemmas}
\label{sub:preliminary_lemmas}

\begin{lem}
      \label{lem:algebra-slope-norm}
  Let $s \in \nint{p}$.
  For any two $\x, \xe \in \bbR^n$ such that $\norm{\Inct \x}_0 \leq s$ we have
  \begin{equation}
    \norms{\Inct \x} - \norms{\Inct \xe}
    \leq
    \sum_{j=1}^s \lambda_j |u|_j^\downarrow - \sum_{j=s+1}^p \lambda_j |u|_j^\downarrow ,
    \le
    \Lambda(\lambda, s) \norm{u}_2 - \sum_{j=s+1}^p \lambda_j |u|_j^\downarrow ,
  \end{equation}
  where $u = \Inct(\xe - \x)$ and
    $\Lambda(\lambda, s) = \Big( \sum_{j=1}^s \lambda_j^2 \Big)^{1/2}$.
    \end{lem}
\begin{proof}
    This is a consequence of \cite[Lemma A.1]{bellec2016slope}. We provide a proof here for completeness.
    The second inequality is simple consequence of the Cauchy-Schwarz inequality. Indeed,
    $(\sum_{j=1}^s \lambda_j |u|_j^\downarrow )^2\le (\sum_{j=1}^s \lambda_j^2 ) (\sum_{j=1}^s (|u|_j^\downarrow)^2) \le \norm{u}_2^2  (\sum_{j=1}^s \lambda_j^2 )$.

    Let $v = \Inct \x$ and $w = \Inct \xe$ and note that $u= w - v$.
    Let $\varphi$ be a permutation of $[p]$ such that $\norms{v} = \sum_{j=1}^s \lambda_j |v|_{\varphi(j)}$.
    The first inequality can be rewritten as
    \begin{equation}
        \norms{v} - \norms{w}
        =
        \norms{v} - \sup_{\tau} \sum_{j=1}^p \lambda_j |w|_{\tau(j)}
        \leq
        \sum_{j=1}^s \lambda_j |u|_j^\downarrow
        -\sum_{j=s+1}^p \lambda_j |u|_j^\downarrow,
        \label{eq:previous-lem-alg-slope-norm}
    \end{equation}
    where the supremum is taken over all permutations $\tau$ of $[p] $.
    We now prove \eqref{eq:previous-lem-alg-slope-norm}.
    Let $\tau$ be a permutation of $[p]$ such that $\varphi(j) = \tau(j)$ for all $j=1,\dots,s$.
    Then by the triangle inequality, we have
    \begin{equation}
        |v|_{\varphi(j)} - |w|_{\varphi(j)}
        =
        |v|_{\varphi(j)} - |w|_{\tau(j)}
        \le
        |u|_{\tau(j)}
    \end{equation}
    for each $j=1,\dots,s$ since $u = w-v$. Furthermore, for each $j>s$, it holds that $v_{\tau(j)} = 0$
    so that $w_{\tau(j)} = u_{\tau(j)}$. Thus,
    \begin{equation}
        \norms{v} - \norms{w}
        \le
        \sum_{j=1}^s \lambda_j |v|_{\varphi(j)}
        -
        \sum_{j=1}^p \lambda_j |w|_{\tau(j)}
        \le
        \sum_{j=1}^s \lambda_j |u|_{\tau(j)}
        -
        \sum_{j=s+1}^p\lambda_j |u|_{\tau(j)}.
    \end{equation}
    It is clear that $\sum_{j=1}^s \lambda_j |u|_{\tau(j)} \le \sum_{j=1}^s \lambda_j |u|_j^\downarrow$.
    Finally, notice that it is always possible to find a permutation $\tau$ such that $(|u|_{\tau(j)})_{j>s}$ is non-decreasing.
    For such choice of $\tau$ we have
    $- \sum_{j=s+1}^p\lambda_j |u|_{\tau(j)}
    \le
        -\sum_{j=s+1}^p \lambda_j |u|_j^\downarrow$.
\end{proof}

\begin{lem}\label{lem:control_lambdas}
For the choice of weights:
$ \forall j \in [p], \lambda_j =  C \sqrt{\log(2p/j)}$,
the following inequalities hold true
\begin{equation}\label{eq:bound-Lambda}
C\sqrt{ s \log(2 p /s)}
\leq \Lambda(\lambda, s)
\leq C\sqrt{s \log(2 e p /s)} 
\enspace.
\end{equation}

\end{lem}

\begin{proof} Reminding Stirling's formula $s \log(s / e) \leq \log(s !) \leq s\log (s)$, one can check that
  \begin{align}
   s \log\left(\frac{2 p}{s}\right)
  & \leq  \sum_{j=1}^s \log \left(\frac{2 p}{j} \right) \leq  s \log(2 p) -  \log(s!) = \log\left(\frac{2 e p}{s}\right) \enspace.
  \end{align}
\end{proof}

\begin{lem}
    \label{lem:strong-convexity}
  Let $z, \eps\in \bbR^n$, $y = z + \eps$, and $\xe$ a solution of Problem~\eqref{eq:gslope}.
  Then, for all $\x \in \bbR^n$,
  \begin{equation}
      \frac 1 2 \left( \norm{ \xe - z}^2 -
      \norm{ \x - z}^2 + \norm{\xe -
  \x}^2 \right)
      \leq \eps^\top(\xe - \x) + n \norms{\Inct \x} - n \norms{\Inct \xe} .
  \end{equation}
\end{lem}
\begin{proof}
    The objective function of the minimization problem~\eqref{eq:gslope} is the sum of two convex functions.
    The first term, \ie the function $\x\to \frac{1}{2n} \norm{y - \x}^2$, is ($1/n$)-strongly convex with respect to the Euclidean norm $\norm{\cdot}$.
    The sum of a 1-strongly convex function and a convex function is 1-strongly convex, and thus by multiplying by $n$ we have
    \begin{equation}
        \frac 1 2 \norm{\xe - y}^2
        +
        n \norms{\Inct \xe}
        \le
        n \mathbf d^T(\xe - \x)
        +
        \frac 1 2 \norm{\x - y}^2
        +
        n \norms{\Inct \x}
        - \frac 1 2 \norm{\x - \xe}^2
    \end{equation}
    for all $\x\in\bbR^n$ and for any $\mathbf d$ in the subdifferential of the objective function \eqref{eq:gslope} at $\xe$.
    Since $\xe$ is a minimizer of $\eqref{eq:gslope}$, we can choose $\mathbf d=0$ in the above display.
    For $\mathbf d = 0$, the previous display is equivalent to the claim of the Lemma.
\end{proof}

\begin{lem}\label{lem:Pi}
Let us suppose that the graph $\G$ has $K$ connected components $C_1,\dots,C_K$. Then,
\begin{equation}
  \ker(\Inct)=\Span(\1_{C_1})\oplus\cdots \oplus \Span(\1_{C_K}) ,
\end{equation}
where for any $k\in \nint{K}$, the vectors $\1_{C_k}\in \bbR^n$ are defined by
\begin{equation}
  (\1_{C_k})_i=
  \begin{cases}
    1 & \text{ if } \, i \in C_k\\
    0 & \text{ otherwise }
  \end{cases},
  \text{ for } i=1,\cdots,|V| \enspace.
\end{equation}

Moreover, the orthogonal projection over $\ker(\Inct)$ is denoted by $\Pi$ and is the component-wise averaging given by
  \begin{equation}
    \left(\Pi (\beta) \right)_i =
    \frac{1}{|C_k|} \sum_{i \in C_k} \beta_i, \text{ where } k \text{ is such that } C_k \ni i, \text{ for } i=1,\dots,n .
  \end{equation}
    Furthermore, if $\G$ is a connected graph then $\ker(\Pi) = \Span(\1_n)$.
\end{lem}

\begin{proof}
  The proof can be done for the simple case of a connected graph (\ie $K=1$), and the result can be generalized by tensorization of graph for $K>1$ components.
  Hence, we assume that $K=1$.
  For any $\beta \in \ker({\Inct})$, the definition of the incidence matrix  yields that for all $(i,j) \in E, \beta_{i}=\beta_{j}$.
  Since all vertices are connected, by recursion all the $\beta_j$'s are identical, and $\beta \in \Span (\1_{n})=\Span (\1_{C_1})$.
  The converse is proved in the same way.
\end{proof}

\begin{lem}[Proposition E.2 in \cite{bellec2016slope}]
    \label{lem:stochastic}
    Let $g_1,\dots,g_p$ be centered Gaussian random variables (not necessarily independent) with variance at most $V>0$.
    Then,
    \begin{equation}
        \bbP\left(
            \max_{j=1,\dots,p} \frac{|g|^\downarrow_j}{\sqrt{V \log(2p/j)}}
            \le 4
        \right) \ge 1/2\enspace.
    \end{equation}
\end{lem}

\bibliographystyle{plain}
\bibliography{biblio}

\end{document}